\documentclass[12pt]{article}
\usepackage{amsmath}
\usepackage{graphicx,psfrag,epsf}
\usepackage{enumerate}
\usepackage{amssymb}
\usepackage{booktabs}
\usepackage{array}
\newcommand{\blind}{0}

\usepackage[authoryear]{natbib}
\usepackage[colorlinks=true, linkcolor=blue, citecolor=blue, urlcolor=blue]{hyperref}
\usepackage[skip=0.3cm]{caption}   % spacing above captions
\usepackage{subcaption}   % modern subfigure support
\usepackage{algorithm}   % for algorithm environment
\usepackage{algpseudocode} % for pseudocode
\usepackage{amsthm}
\usepackage{float}
\newtheorem{theorem}{Theorem}[section]
\newtheorem{proposition}{Proposition}[section]

\newtheorem{assumption}{Assumption}[section]
\begin{document}

\def\spacingset#1{\renewcommand{\baselinestretch}%
{#1}\small\normalsize} \spacingset{1}

%%%%%%%%%%%%%%%%%%%%%%%%%%%%%%%%%%%%%%%%%%%%%%%%%%%%%%%%%%%%%%%%%%%%%%%%%%%%%%

\date{} % Removes the date
\if0\blind
{
  \title{\bf A Replica Exchange Markov Chain Monte Carlo Method for Disconnected Implicit Manifolds via Tubular Relaxation}
  \author{
    Xuyuan Wang$^{\dagger}$\thanks{Corresponding author. Email: \texttt{xuyuan@ualberta.ca}; ORCID: [0009-0004-4470-8200].}
    \and
    Donglin Han\thanks{Email: \texttt{donglin3@ualberta.ca}; ORCID: [0009-0003-8502-0317].}
    \\ \medskip
    \textsuperscript{$\dagger$}These authors contributed equally.
    \\ \medskip
    Department of Mathematical and Statistical Sciences, University of Alberta
  }
  \maketitle
} \fi

\if1\blind
{
  \bigskip
  \bigskip
  \bigskip
  \begin{center}
    {\LARGE\bf A Replica Exchange Markov Chain Monte Carlo Method for Disconnected Implicit Manifolds via Tubular Relaxation}
  \end{center}
  \medskip
} \fi

\begin{abstract}
Markov chain Monte Carlo (MCMC) methods provide powerful framework for sampling unknown probability measures across a wide range of scientific applications. In some settings, the target distribution is supported on a lower-dimensional submanifold of Euclidean space defined by nonlinear constraints, motivating the development of constrained Hamiltonian Monte Carlo (CHMC) methods. Most existing CHMC algorithms rely on the assumption that the implicit manifold is connected, allowing local constrained integrators such as RATTLE to explore the posterior ergodically. In practice, this assumption is occasionally violated due to complex geometric structures induced by nonlinear constraints of a model.
We propose a replica exchange MCMC framework that couples a constrained chain evolving on the implicit manifold with a relaxed auxiliary chain defined in a tubular neighborhood of the constraint. The relaxed chain enables transitions between disconnected components. We show that the resulting algorithm enables sampling from a broader class of implicit manifolds, including those with disconnected components. We prove that the proposed sampler satisfies detailed balance, irreducibility, ergodicity, and convergence.
We also demonstrate its effectiveness on examples from molecular and biological dynamical systems.

\end{abstract}

\noindent%

\spacingset{1.45}
\section{Introduction}\label{sec1}
Markov Chain Monte Carlo (MCMC) methods are an important class of algorithms for sampling from complex probability measures when only an unnormalized target density is available. Thanks to this flexibility, MCMC has been widely adopted across scientific disciplines, including molecular dynamics \citep{J1}, statistical physics \citep{J2}, and the calibration of biological dynamical models \citep{J3}. In these applications, MCMC enables uncertainty quantification and statistical inference in settings where direct sampling is infeasible due to high dimensionality, nonlinear structure, or intractable normalizing constants.

In a number of important settings, the target probability measure is not supported on the full ambient Euclidean space but rather on a lower-dimensional manifold that is defined implicitly through holonomic constraints. This situation arises naturally in molecular dynamics, where one is often interested in Monte Carlo estimation of equilibrium measures subject to fixed bond lengths, bond angles, or other constraints on particle positions \citep{J4,J5}. These constraints define a configuration space that is a smooth submanifold of the ambient space, typically specified as the zero level set of a nonlinear constraint function. Sampling methods that fail to keep this geometric structure may suffer from poor efficiency or produce invalid samples.

A similar challenge arises in Bayesian inverse problems for parameter calibration in biological and epidemiological models governed by systems of differential equations. In such models, structural or practical non-identifiability is common: the available data are insufficient to uniquely determine all model parameters, even in the absence of observation noise \citep{J6}. Instead, only certain identifiable combinations of parameters can be inferred from the observable quantities.
From a geometric perspective, the inverse image of a calibrated identifiable parameter combination typically forms a lower-dimensional manifold embedded in the original parameter space. Consequently, recovering samples in the original parameterization from the calibrated parameter combination value requires algorithms that can accurately and efficiently explore the implicit manifold at the calibrated identifiable quantities \citep{Mlift,J7,J8}.

These considerations have motivated a growing body of work on numerical methods for sampling from probability measures constrained to implicit manifolds. One class of approaches is based on projecting stochastic differential equations, such as Langevin dynamics, onto the tangent space of the constraint manifold \citep{J4, J10, J9}. These methods employ local projection schemes that correct unconstrained dynamics by orthogonally projecting proposed moves back onto the manifold. 
Another line of research is rooted in constrained Hamiltonian dynamics. In this framework, auxiliary momentum variables are introduced and the target distribution is augmented to a phase-space measure, allowing the use of Hamiltonian Monte Carlo (HMC) type proposals \citep{M2, M3}. The evolution of the system is governed by Hamilton’s equations subject to holonomic constraints, ensuring that both position and momentum remain consistent with the manifold structure. Numerical integrators such as the SHAKE and RATTLE algorithms \citep{J11} enforce the constraints at each integration step while preserving symplectic structure and time reversibility. When combined with a Metropolis correction \citep{J17}, these methods yield valid MCMC algorithms that can achieve long-range proposals and improved mixing compared to purely diffusive schemes.

Despite these advances, sampling on implicitly defined manifolds remains challenging, particularly when the holonomic constraints induce disconnected implicit manifold \citep{J12}. Such disconnected manifolds arise naturally in several applications. In molecular dynamics, for example, nonlinear bond-length and bond-angle constraints can induce multiple disconnected configuration components corresponding to distinct molecular conformations \citep{J13}. Similarly, in inverse problems with parameter non-identifiability, identifiable parameter combinations are often obtained through algebraic mappings from the original parameter space \citep{J14}. As a result of algebraic symmetries, sign ambiguities, or polynomial invariants, the inverse image of an identifiable parameter value may decompose into multiple disconnected branches.

Disconnectedness poses a fundamental difficulty for most existing MCMC algorithms. Methods based on local projection of Langevin dynamics or constrained Hamiltonian systems generate proposals by evolving dynamics in the local tangent space of the manifold. These algorithms are inherently local: step sizes must be sufficiently small to ensure numerical stability, constraint satisfaction, and reversibility \citep{M5}. Consequently, transitions between disconnected components, are generally impossible. This limitation leads to a failure to explore the full support of the target distribution when the manifold is disconnected.

In this work, we propose a new MCMC sampling method on implicitly defined manifolds that addresses the challenge of disconnected components. Our approach is based on two key ideas. First, instead of restricting proposals strictly to the constraint manifold, we exploit the existence of local tubular neighborhoods around the manifold. Within such a neighborhood, points in the ambient space admit a well-defined projection onto the manifold, allowing us to construct proposal kernels that move in the surrounding space while remaining coupled to the geometry of the constraint. 
Second, to enable transitions between disconnected components, we apply the local proposal mechanism in a replica exchange framework \citep{J15}. By introducing a sequence of auxiliary distributions that progressively relax the constraint enforcement, we allow the chains to explore regions of the ambient space that connect isolated components of the manifold. We show that under regularity assumptions on the constraint mapping, the resulting Markov chain satisfies detailed balance with respect to the target measure and is ergodic over the full support of the manifold. 

The paper is organized as follows. In section \ref{sec2}, we present some of the related works to this matter. In section \ref{sec3}, we review some of the key concepts in the literature which forms a basis for our method, and introduce our methods. Section \ref{sec4} provides the numerical examples to demonstrate the effectiveness of our method. Lastly in section \ref{sec5} we provide a summary of the method and the potential  future directions.

\section{Related Work}\label{sec2}
The foundation of this paper builds on the well-established constrained Hamiltonian Monte Carlo (CHMC) methodology, which has been extensively studied in the literature. Representative works include \citet{J4}, who introduced a constrained hybrid Monte Carlo algorithm for sampling on manifolds arising in free energy calculations; \citet{M2}, who developed numerical schemes for constrained Langevin and Hamiltonian dynamics; \citet{M3}, who formulated MCMC methods on implicitly defined manifolds. These approaches, along with the work of \citet{M4}, follow the projection-based framework of \citet{M5}. This framework established the core conditions for sampling on manifolds while ensuring the process remains reversible.

In contrast, the problem of sampling from disconnected constraint sets has received relatively limited attention. Some works have considered related settings. For example, approaches rooted in algebraic geometry have been proposed to generate random samples from algebraic constraint sets by exploiting their underlying algebraic structure \citep{RPAM}. While such methods can handle disconnected solution sets, they are typically problem specific and do not extend to sampling from general probability measures defined on disconnected implicit manifolds.

\section{Background}\label{sec3}
We begin by introducing the basic mathematical definitions and notation required to discuss MCMC methods on general implicit manifolds. An implicit manifold $\mathcal{M}$ of co-dimension $m$ embedded in the ambient space $\mathbb{R}^n$ is defined as
\[
\mathcal{M} := \{ q \in \mathbb{R}^n : \xi(q) = 0 \},
\]
where the constraint map $\xi : \mathbb{R}^n \to \mathbb{R}^m$ is assumed to be analytically available. For the purposes of this work, we make the following assumptions on the constraint function $\xi$.
\begin{assumption}[Constraint function]\label{ass:xi}
The constraint function $\xi$ satisfies:
\begin{enumerate}
    \item[(i)] $\xi \in C^\infty(\mathbb{R}^n, \mathbb R^m)$, i.e., it is smooth on the ambient space;
    \item[(ii)] $\nabla \xi(q):=[\nabla\xi_1(x),\nabla\xi_2(x),\cdots,\nabla\xi_m(x)]\in\mathbb R^{n\times m}$ is uniformly of full rank on $\mathcal{M}$, i.e., $\nabla \xi(q)^\top \nabla \xi(q) \succeq c I_m$ for some constant $c > 0$ and all $q \in \mathcal{M}$;
    \item[(iii)] The Hessian $\nabla^2 \xi_i(q)$ is uniformly bounded on 
    $\mathcal{M}$, i.e.\ there exists $C > 0$ such that 
    $\|\nabla^2 \xi_i(q)\| \le C$ for all $q \in \mathcal{M}$ and 
    $i = 1,\ldots,m$.
\end{enumerate}
\end{assumption}
Under these assumptions, $\mathcal{M}$ is a smoothly embedded submanifold of $\mathbb{R}^n$.
The standard Euclidean inner product on $\mathbb{R}^n$ induces a Riemannian metric on $\mathcal{M}$, which in turn defines a Riemannian volume measure $\sigma_{\mathcal{M}}(dq)$ on the manifold. This measure is singular with respect to the $n$-dimensional Lebesgue measure on $\mathbb{R}^n$, i.e., $\mathcal{M}$ has zero Lebesgue measure in the ambient space.
The goal of an MCMC algorithm in this setting is to sample from a probability measure supported on $\mathcal{M}$ of the form
\begin{equation}\label{eq:target_measure}
\mu(dq) 
= \frac{e^{-V(q)} \, \sigma_{\mathcal{M}}(dq)}{\int_{\mathcal{M}} e^{-V(q)} \, \sigma_{\mathcal{M}}(dq)},
\end{equation}
where $V : \mathbb{R}^n \to \mathbb{R}$ is a potential energy function and $\sigma_{\mathcal{M}}$ is the surface measure induced by the ambient Euclidean metric on $\mathcal{M}$.
We make the following assumptions on the potential function $V$.
\begin{assumption}[Potential function]\label{ass:potential}
The potential function $V$ satisfies:
\begin{enumerate}
    \item[(i)] $V \in C^\infty(\mathbb{R}^n,\mathbb R)$, i.e., it is smooth on the ambient space;
    \item[(ii)] $V$ is bounded from below, i.e., there exists a constant $C \in \mathbb{R}$ such that $V(x) \ge C$ for all $x \in \mathbb{R}^n$;
    \item[(iii)] $V(x) \to \infty$ as $\|x\| \to \infty$, and $\int_{\mathbb{R}^n} e^{-V(x)} dx < \infty$.
\end{enumerate}
\end{assumption}
These conditions ensure that the target measure \eqref{eq:target_measure} is well-defined and normalizable. In physical terms, $V$ represents the potential energy of a Hamiltonian system, while in Bayesian inverse problems, $V$ corresponds to the negative log-posterior density.

\subsection{Hamiltonian Monte Carlo}\label{sec3.1}

Hamiltonian Monte Carlo (HMC)  is a widely used method for sampling probability measures defined on the ambient space $\mathbb{R}^n$. 
The key idea of HMC is to augment the configuration variable $q \in \mathbb{R}^n$ with an auxiliary momentum variable $p \in \mathbb{R}^n$, thereby lifting the problem to the phase space $(q,p) \in \mathbb{R}^{2n}$. 
A joint probability measure is then defined via the Hamiltonian
\[
H(q,p) = V(q) + \frac{1}{2}\|p\|_2^2,
\]
where $V(q)$ is the potential energy and $\frac{1}{2}\|p\|_2^2$ is the kinetic energy, corresponding to a standard Gaussian distribution on the momentum variable. The associated Hamiltonian dynamics are given by
\[
\dot{q} = \nabla_p H(q,p) = p, 
\qquad
\dot{p} = -\nabla_q H(q,p) = -\nabla V(q),
\]
which preserve the Hamiltonian $H$ and the canonical volume measure on phase space. 
These dynamics are time-reversible and symplectic, properties that make them well suited for constructing efficient MCMC proposal mechanisms \citep{M3}.
In practice, the continuous time Hamiltonian flow cannot be solved exactly and is approximated using a symplectic integrator, most commonly the Störmer--Verlet (leapfrog) scheme \citep{J16}. 
The resulting proposal is then accepted or rejected using a Metropolis--Hastings correction step \citep{J17} to ensure exact invariance of the target distribution. A standard HMC algorithm is summarized below.
\begin{algorithm}[htbp]
\caption{Hamiltonian Monte Carlo}
\label{alg:HMC}
\begin{algorithmic}[1]
\Require Initial position $q^{(0)} \in \mathbb{R}^n$, step size $\varepsilon > 0$, number of leapfrog steps $L$
\For{$k = 0,1,2,\dots$}
    \State Sample momentum $p^{(k)} \sim \mathcal{N}(0, I_n)$
    \State Set $(q_0, p_0) = (q^{(k)}, p^{(k)})$
    \For{$\ell = 0$ to $L-1$}
        \State $p_{\ell+\frac12} = p_\ell - \frac{\varepsilon}{2} \nabla V(q_\ell)$
        \State $q_{\ell+1} = q_\ell + \varepsilon \, p_{\ell+\frac12}$
        \State $p_{\ell+1} = p_{\ell+\frac12} - \frac{\varepsilon}{2} \nabla V(q_{\ell+1})$
    \EndFor
    \State Propose $(q',p') = (q_L,-p_L)$
    \State Accept with probability
    \[
    \alpha = \min\left\{1, \exp\bigl(-H(q',p') + H(q^{(k)},p^{(k)})\bigr)\right\}
    \]
    \State Set $q^{(k+1)} = q'$ with probability $\alpha$, otherwise $q^{(k+1)} = q^{(k)}$
\EndFor
\end{algorithmic}
\end{algorithm}
It can be shown that Algorithm~\ref{alg:HMC} admits the Hamiltonian $H$ as its stationary distribution. The target probability measure is then obtained by marginalizing out the momentum variable $p$.
\subsection{Constrained Hamiltonian Monte Carlo}\label{sec3.2}

While standard HMC is designed for probability measures supported on the ambient space $\mathbb{R}^n$, it cannot be directly applied to sample from probability measures supported on an implicit manifold of the form equation \eqref{eq:target_measure}. As discussed earlier, such a target measure is supported on a set of zero Lebesgue measure in the ambient space, and unconstrained Hamiltonian dynamics will almost surely leave the manifold after any finite integration steps. 
Therefore, appropriate constraints must be imposed on the Hamiltonian dynamics to ensure that the sampling process remains on the manifold $\mathcal{M}$. A common approach is to introduce Lagrange multipliers that enforce the constraint $\xi(q)=0$ along the Hamiltonian flow \citep{M4}. Consider the constrained Hamiltonian system
\begin{equation}\label{eq:CHMC_dyn}
\begin{cases}
&\dot{q} = p, \\
&\dot{p} = -\nabla V(q) + \nabla \xi(q)\,\lambda, \\
& 0 = \xi(q),
\end{cases}
\end{equation}
where $\lambda \in \mathbb{R}^m$ is a vector of Lagrange multipliers chosen so that the constraint $\xi(q)=0$ is satisfied for all time. 
In addition to the position constraint, consistency of the dynamics requires the momentum to lie in the cotangent space of the manifold, namely $T^*_q\mathcal{M} := \{p\in\mathbb{R}^n: \nabla \xi(q)^\top p = 0\}$,
so that the flow remains on $\mathcal{M}$.
The constrained Hamiltonian dynamics \eqref{eq:CHMC_dyn} preserve the Hamiltonian
as well as the canonical measure restricted to the cotangent bundle $T^*\mathcal{M}:=\{(p,q)\in\mathbb{R}^{2n}: \xi(q)=0, \nabla \xi(q)^\top p = 0 \}$. 
As in the unconstrained case, these dynamics can be approximated numerically.
A standard numerical integrator for constrained Hamiltonian systems is the RATTLE scheme, which is a symplectic and time-reversible extension of the leapfrog integrator \citep{J11}. 
At each integration step, RATTLE enforces both the position and momentum constraints through the solution of nonlinear equations for the Lagrange multipliers.
A single RATTLE step with step size $\varepsilon$ is given by
\begin{equation}\label{eq:RATTLE}
\begin{cases}
&p_{n+\frac12} = p_n - \frac{\varepsilon}{2}\nabla V(q_n) + \nabla \xi(q_n)\,\lambda_{n+\frac12}, \\
&q_{n+1} = q_n + \varepsilon p_{n+\frac12}, \\ &\xi(q_{n+1}) = 0, \\
&p_{n+1} = p_{n+\frac12} - \frac{\varepsilon}{2}\nabla V(q_{n+1}) + \nabla \xi(q_{n+1})\,\lambda_{n+1}, \\
&\nabla \xi(q_{n+1})^\top p_{n+1} = 0,
\end{cases}
\end{equation}
where the Lagrange multipliers $\lambda_{n+\frac12}, \lambda_{n+1} \in \mathbb{R}^m$
are chosen so that the position and momentum constraints are satisfied. Using this constrained Hamiltonian integrator, a constrained Hamiltonian Monte Carlo
(CHMC) algorithm can be obtained by replacing the classical leapfrog integrator Algorithm~\ref{alg:HMC} with the RATTLE integrator \eqref{eq:RATTLE}.
\begin{algorithm}[htbp]
\caption{Constrained Hamiltonian Monte Carlo}
\label{alg:CHMC}
\begin{algorithmic}[1]
\Require Initial position $q^{(0)} \in \mathcal{M}$,
step size $\varepsilon > 0$,
number of RATTLE steps $L$
\For{$k = 0,1,2,\dots$}
    \State Sample $\tilde p^{(k)} \sim \mathcal{N}(0,I_n)$
    \State Project momentum onto the cotangent space:
    \begin{equation}\label{mproj}
    p^{(k)} = \Pi_{q^{(k)}} \tilde p^{(k)}, 
    \quad
    \Pi_q = I - \nabla \xi(q)
    \bigl(\nabla \xi(q)^\top \nabla \xi(q)\bigr)^{-1}
    \nabla \xi(q)^\top
    \end{equation}
    \State Set $(q_0,p_0)=(q^{(k)},p^{(k)})$
    \For{$\ell = 0$ to $L-1$}
        \State Compute $(q_{\ell+1},p_{\ell+1})$
        using one RATTLE step \eqref{eq:RATTLE}
        \State \textbf{abort trajectory if reversibility is violated}
    \EndFor
    \State Propose $(q',p')=(q_L,-p_L)$
    \State Accept with probability
    \[
    \alpha = \min\left\{
    1,
    \exp\bigl(
    -H(q',p') + H(q^{(k)},p^{(k)})
    \bigr)
    \right\}
    \]
    \State Set $q^{(k+1)} = q'$ with probability $\alpha$,
    otherwise $q^{(k+1)} = q^{(k)}$
\EndFor
\end{algorithmic}
\end{algorithm}
In general, the momentum constraint $\nabla \xi(q)^\top p = 0$ is linear in $p$ and
can be interpreted as an orthogonal projection onto the cotangent space
$T^*_{q}\mathcal{M}$, which is well defined whenever $\nabla \xi(q)$ has full
rank. In contrast, the position constraint $\xi(q) = 0$ requires a nonlinear
projection onto the manifold $\mathcal{M}$. Depending on the step size $\varepsilon$
and the local curvature of $\mathcal{M}$ at $q_n$, such a projection may have multiple
solutions or no solution at all. Moreover, even when a projection exists, it need not
be an involution, which may break the time reversibility of the numerical integrator
and invalidate the Metropolis--Hastings acceptance probability.
As a result, a reversibility check of the position projection is required to guarantee the validity of the CHMC algorithm \citep{M5}. Such a check can be carried out by running the integrator \eqref{eq:RATTLE} backward in time at each step and verifying that the original state is recovered. If reversibility fails, the entire trajectory is discarded and the chain remains at the current state (Line~7 of Algorithm~\ref{alg:CHMC}). For a detailed
analysis of the mathematical properties of position projection, we refer the reader to the mathemtical analysis presented in \citet{M4}.

\subsection{Challenges}\label{sec3.3}
In the general setting considered here, we assume only that the implicit manifold
$\mathcal{M}$ is smooth. However, smoothness alone does not imply connectedness.
For example, let $\xi : \mathbb{R} \to \mathbb{R}$ be defined by
\(
\xi(q) = q^2 - 1.
\)
Then $\xi$ is smooth, and the resulting implicit manifold
\(
\mathcal{M} = \{ q \in \mathbb{R} : q^2 = 1 \} = \{-1,\,1\}
\)
is a smooth but disconnected manifold consisting of two disjoint components.
More generally, disconnected implicit manifolds arise naturally in nonlinear
constraint systems.
In general, there is not a simple criterion that guarantees
connectedness of $\mathcal{M}$, nor a method to determine the number of its connected
components a priori \citep{J18}. As a result, imposing connectedness as an assumption in the
algorithmic design is often unrealistic and excludes many practically relevant models.

As discussed in Section \ref{sec3.2}, CHMC relies on sufficiently small step
sizes $\varepsilon$ to ensure the existence, uniqueness, and reversibility of the
projection steps. The resulting dynamics are local and are
confined to a single connected component of $\mathcal{M}$, making transitions between
disconnected components impossible. 
In the following sections, we introduce auxiliary tubular relaxed chains combined with a
replica exchange strategy for sampling on disconnected implicit manifolds. We present
a detailed construction of the exchange mechanism and show that the resulting method enables
transitions between disconnected components while preserving the exact target
probability measure $\mu$ defined in equation \eqref{eq:target_measure}.

 \section{Method}\label{sec4}

We begin by recalling the notion of a tubular neighborhood for a smooth implicit
manifold $\mathcal{M}$.
For any point $q \in \mathcal{M}$, the normal space at $q$ is defined by
\(
N_q \mathcal{M}
:= \operatorname{Range}\bigl(\nabla \xi(q)\bigr)
= \{ \nabla \xi(q) v : v \in \mathbb{R}^m \}.
\)
The normal bundle of $\mathcal{M}$ is then
\(
N \mathcal{M}
:= \{ (q, v) \in \mathbb{R}^{n+m} : \xi(q) = 0,\nabla\xi(q)v \in N_q \mathcal{M} \}.
\)
A fundamental geometric result guarantees the existence of local coordinates
around $\mathcal{M}$ obtained by projecting nearby points along normal directions.
We recall the classical local tubular neighborhood theorem \citep{J19} for smooth embedded
manifolds in the following Theorem \ref{thm:tubular_local}.
\begin{theorem}[Local tubular neighborhood]
\label{thm:tubular_local}
For every $q \in \mathcal{M}$, there exists an open neighborhood
$U_q \subset \mathbb{R}^n$ of $q$ and a smooth projection mapping
\(
P : U_q \to \mathcal{M}
\)
such that:
\begin{enumerate}
    \item[(i)] $P(x) = x$ for all $x \in U_q \cap \mathcal{M}$;
    \item[(ii)] for every $x \in U_q$, $P(x)$ is the unique point in $\mathcal{M}$ minimizing
    $\|x - y\|$ over $y \in \mathcal{M}$;
    \item[(iii)]  The mapping
\[
x \longmapsto \bigl(P(x), v(x)\bigr),
\quad s.t. \quad x=P(x)+\nabla\xi\bigl(P(x)\bigr)v(x) 
,
\]
    defines a diffeomorphism between $U$ and an open subset of the normal bundle
    $N \mathcal{M}$.
\end{enumerate}
In particular, the projection $P$ is uniquely defined and smooth on $U$.
\end{theorem}
By Theorem \ref{thm:tubular_local}, every point sufficiently close to $\mathcal{M}$ admits a unique
decomposition into a base point on the manifold and a normal displacement.
Within a tubular neighborhood $U_q$, the projection $q_x = P(x)$ can be characterized
as the solution of the nonlinear system
\begin{equation}
\label{eq:projection_system}
\begin{cases}
x  = q_x + \nabla \xi(q_x) v_x, \\
\xi(q_x) = 0,
\end{cases}
\end{equation}
with unknowns $(q_x,v_x) \in \mathbb{R}^{n+m}$.
Since $\nabla\xi(q)$ has full rank $m$ for all $q\in\mathcal M$, 
the Jacobian of system \eqref{eq:projection_system} with respect to $(q_x,v_x)$ 
is nonsingular at any solution. By the implicit function theorem, 
for all $x$ sufficiently close to $\mathcal M$, 
there exists a unique solution $(q_x,v_x)$ depending smoothly on $x$.
Standard numerical methods, such as Newton's method \citep{J20},  
can therefore be applied to solve the nonlinear system 
\eqref{eq:projection_system}. To formally characterize the set of ambient points that admit such a unique
normal decomposition, we introduce the following global tubular domain.
\begin{proposition}[Uniform global tubular domain]
\label{prop:maximal_tube}
Let $\mathcal{M} = \{q \in \mathbb{R}^n : \xi(q)=0\}$ be an implicit manifold whose constraint function satisfies Assumption \ref{ass:xi}. Then, there exists a uniform tubular radius $\tau > 0$ such that the global tubular domain, defined as
$$
\mathcal{T}_\tau(\mathcal{M}) := \{x \in \mathbb{R}^n : \mathrm{dist}(x, \mathcal{M}) < \tau\}
$$
satisfies the following properties:
\begin{enumerate}
    \item[(i)] $\mathcal{T}_\tau(\mathcal{M})$ is open in $\mathbb{R}^n$ and therefore Lebesgue measurable.
    \item[(ii)] There exists a unique smooth projection 
    $$
    P: \mathcal{T}_\tau(\mathcal{M}) \to \mathcal{M}
    $$ 
    such that for every $x \in \mathcal{T}_\tau(\mathcal{M})$, $P(x)$ is the unique nearest point to $x$ in $\mathcal{M}$.
    \item[(iii)] Every $x \in \mathcal{T}_\tau(\mathcal{M})$ admits a unique normal bundle decomposition \eqref{eq:projection_system}, and the mapping
    $$
    \Phi(x) = (q_x, v_x), \quad \text{s.t.} \quad x = q_x + \nabla\xi(q_x)v_x
    $$
    is a diffeomorphism between $\mathcal{T}_\tau(\mathcal{M})$ and the open $\tau$-neighborhood of the zero section in the normal bundle $N\mathcal{M}$.
\end{enumerate}
\end{proposition}
\begin{proof}
Since $\nabla\xi(q) \in \mathbb{R}^{n \times m}$ with $\nabla\xi(q) = [\nabla\xi_1(q), \ldots, \nabla\xi_m(q)]$ by
Assumption~\ref{ass:xi}(i), the normal space at $q \in \mathcal{M}$ is
$N_q\mathcal{M} = \mathrm{Im}(\nabla\xi(q)) \subset \mathbb{R}^n$, which has dimension $m$ since
$\nabla\xi(q)$ has full column rank by Assumption~\ref{ass:xi}(ii). The condition
$\nabla\xi(q)^\top \nabla\xi(q) \succeq c I_m$ implies $\sigma_{\min}(\nabla\xi(q)) \geq \sqrt{c} > 0$,
so $\nabla\xi(q) : \mathbb{R}^m \to \mathbb{R}^n$ is injective. The cotangent space
$T_q^*\mathcal{M}$ satisfies $T_q^*\mathcal{M} \perp N_q\mathcal{M}$
and $\mathbb{R}^n = T_q^*\mathcal{M} \oplus N_q\mathcal{M}$ orthogonally. Consider the inverse map
$\Phi^{-1} : N\mathcal{M} \to \mathbb{R}^n$ by
\begin{equation}\nonumber
    \Phi^{-1}(q,v) := q + \nabla\xi(q)v,
\end{equation}
where $v \in \mathbb{R}^m$ parametrizes the normal displacement $\nabla\xi(q)v \in N_q\mathcal{M}$. For
$(\delta q, \delta v) \in T_q^*\mathcal{M} \times \mathbb{R}^m$, differentiating along a curve
$(q(t), v(t))$ with $(q(0),v(0)) = (q,v)$ and $(\dot{q}(0), \dot{v}(0)) = (\delta q, \delta v)$
and applying the product rule gives
\begin{equation}\label{eq:dphi-1}
    D\Phi^{-1}(q,v)(\delta q, \delta v)
    = \bigl(I_n + E(q,v)\bigr)\delta q + \nabla\xi(q)\,\delta v,
    \qquad
    E(q,v) := \sum_{i=1}^m v_i\,\nabla^2\xi_i(q).
\end{equation}
At $v = 0$ this reduces to $D\Phi^{-1}(q,0)(\delta q, \delta v) = \delta q + \nabla\xi(q)\delta v$.
This map is injective. Suppose that $\delta q + \nabla\xi(q)\delta v = 0$, because $\delta q \in T_q^*\mathcal{M}$ and the image of $\nabla\xi(q)$ spans the normal space $N_q\mathcal{M}$, these two components belong to orthogonal subspaces of the ambient space. Projecting the equation onto $T_q^*\mathcal{M}$ and $N_q\mathcal{M}$ separately yields $\delta q = 0$ and $\nabla\xi(q)\delta v = 0$. By the injectivity of $\nabla\xi(q)$, it follows that $\delta v = 0$. Thus, the differential is injective at $v=0$. Since the
domain $T_q^*\mathcal{M} \times \mathbb{R}^m$ has dimension $(n-m) + m = n$, this is an isomorphism.
For any $(\delta q, \delta v)$ with $\|(\delta q,\delta v)\| = 1$, the orthogonality
$\delta q \perp \nabla\xi(q)\delta v$ gives
\begin{equation}\nonumber
    \|D\Phi^{-1}(q,0)(\delta q,\delta v)\|^2
    = \|\delta q\|^2 + \|\nabla\xi(q)\delta v\|^2
    \geq \|\delta q\|^2 + c\|\delta v\|^2
    \geq \min(1,c),
\end{equation}
so $\sigma_{\min}(D\Phi^{-1}(q,0)) \geq \sqrt{\min(1,c)} > 0$ uniformly in $q \in \mathcal{M}$.
For $v \neq 0$, Assumption~\ref{ass:xi}(iii) and the Cauchy--Schwarz inequality give
$\|E(q,v)\| \leq C\sqrt{m}\,\|v\|$, so
\begin{equation}\nonumber
    \|D\Phi^{-1}(q,v)(\delta q,\delta v)\|
    \geq \bigl(\sqrt{\min(1,c)} - C\sqrt{m}\,\|v\|\bigr)\|(\delta q,\delta v)\|.
\end{equation}
Setting
\begin{equation}\nonumber
    \tau_1 := \frac{\sqrt{\min(1,c)}}{2C\sqrt{m}} > 0,
\end{equation}
we obtain $\sigma_{\min}(D\Phi^{-1}(q,v)) \geq \frac{\sqrt{\min(1,c)}}{2} > 0$ for all
$q \in \mathcal{M}$ and $\|v\| \leq \tau_1$, uniformly in $q$. The Inverse Function Theorem
therefore guarantees that $\Phi^{-1}$ is a smooth local diffeomorphism on $N^{\tau_1}\mathcal{M}:=\{(q,v)\in N\mathcal{M}: \Vert v\Vert<\tau_1\}$.
To promote this to a global diffeomorphism, we establish a positive lower bound on the maximal tubular neighborhood radius. For a unit vector $\hat{v} \in \mathbb{R}^m$, define the scalar function $(\hat{v}\cdot\xi)(q) := \sum_{i=1}^m \hat{v}_i\xi_i(q)$. From standard submanifold geometry \citep[Chapter 8]{Lee2018}, the principal curvatures $\kappa_j(q)$ of the implicitly defined manifold $\mathcal{M}$ in the unit normal direction $\nu = \nabla\xi(q)\hat{v}/\|\nabla\xi(q)\hat{v}\|$ are bounded by the spectral norm of the Hessian of the constraint function divided by the magnitude of the normal gradient. By Assumptions~\ref{ass:xi}(ii)--(iii), we have
\begin{equation}\nonumber
    |\kappa_j(q)|
    \leq \frac{\|\nabla^2(\hat{v}\cdot\xi)(q)\|}{\sigma_{\min}(\nabla\xi(q))}
    \leq \frac{C\sqrt{m}}{\sqrt{c}}.
\end{equation}
This establishes a uniform maximum principal curvature $\kappa_{\max} := C\sqrt{m}/\sqrt{c} < \infty$. By a classical result of Federer \cite[Theorem~4.18]{Federer1959}, this uniform curvature bound ensures the reach of the manifold is strictly positive
\begin{equation}\nonumber
\mathrm{reach}(\mathcal{M}) \geq \frac{1}{\kappa_{\max}} = \frac{\sqrt{c}}{C\sqrt{m}} > 0.
\end{equation}
By tubular neighborhood theorem \ref{thm:tubular_local}, for any $x$ with $\mathrm{dist}(x, \mathcal{M}) < \mathrm{reach}(\mathcal{M})$, the nearest point in $\mathcal{M}$ is unique, and the residual $x - P(x)$ lies entirely within the normal space $N_{P(x)}\mathcal{M}$. Setting the uniform tubular radius to
\begin{equation}\nonumber
    \tau := \min\!\left(\frac{\sqrt{\min(1,c)}}{2C\sqrt{m}},\; \frac{\sqrt{c}}{C\sqrt{m}}\right) > 0
\end{equation}
ensures the map $\Phi^{-1}$ is simultaneously a local diffeomorphism and globally injective on
$N^\tau\mathcal{M}$, hence a smooth embedding. Its image is exactly $\mathcal{T}_\tau(\mathcal{M})$.
For any $x \in \mathcal{T}_\tau(\mathcal{M})$, the nearest point $q_x = P(x)$ is well-defined, and the normal displacement can be written uniquely as
\[
x - q_x = \nabla \xi(q_x)\, v_x, \quad \|v_x\| < \tau,
\]
so $x = \Phi^{-1}(q_x, v_x)$. Conversely, for $(q,v)\in N^\tau\mathcal{M}$,
\[
\mathrm{dist}(\Phi^{-1}(q,v),\mathcal{M}) \le \|\nabla\xi(q)v\| < \tau,
\]
so $\Phi^{-1}(q,v)\in \mathcal{T}_\tau(\mathcal{M})$.
Define
$\Phi := (\Phi^{-1}|_{N^\tau\mathcal{M}})^{-1}$, properties~(i)--(iii) follow immediately.
$\mathcal{T}_\tau(\mathcal{M})$ is the image of an open set under a smooth embedding, hence open
and Lebesgue measurable. The smooth projection $P(x) = q_x$ is unique by the positive reach
argument, and $\Phi$ is a smooth diffeomorphism by construction, providing the explicit normal
bundle decomposition $x = q_x + \nabla\xi(q_x)v_x$ with uniformly controlled radius $\tau > 0$.
\end{proof}
As discussed in Section~\ref{sec3.3}, a central difficulty in sampling from
the target measure
\(
\mu(dq) \propto \exp\{-V(q)\}\,\sigma_{\mathcal M}(dq),
\)
is that it is supported on the embedded manifold $\mathcal M$,
which has zero Lebesgue measure in $\mathbb R^n$.
Moreover, when $\mathcal M$ consists of multiple disconnected components,
direct sampling methods may struggle to explore the full support.
Proposition~\ref{prop:maximal_tube} establishes a one-to-one
correspondence between ambient points in the tubular domain
$\mathcal T_\tau(\mathcal M)$ and pairs $(q,v)\in\mathbb R^{n+m}$ in the normal bundle.
Since $\mathcal T_\tau(\mathcal M)$ is an open subset of $\mathbb R^n$
with positive Lebesgue measure,
this suggests an equivalent sampling strategy:
instead of sampling directly on $\mathcal M$,
we sample in its tubular neighborhood and subsequently project onto the manifold.

\subsection{Tubular Relaxation}
We start by introducing the tubular relaxation of
the constrained target measure.
Rather than enforcing the hard constraint $\xi(x)=0$,
we allow the Markov chain to explore a neighborhood of the manifold
while penalizing deviations in the normal direction through a quadratic
regularization. We define
\begin{equation}
\label{eq:tubular_relaxation}
\pi_\varepsilon(dx)
\;\propto\;
\exp\!\left(
- V(x)
- \frac{1}{\varepsilon}\|\xi(x)\|^2
\right) dx,
\end{equation}
where $\varepsilon>0$ is a relaxation parameter.
The penalty term $\varepsilon^{-1}\|\xi(x)\|^2$
induces Gaussian concentration in the normal direction.
Under the normal bundle decomposition $\Phi(x)=(q_x,v_x)$,
where $x \in \mathcal{T}_\tau(\mathcal{M})$,
a second-order Taylor expansion of $\xi$ at $q_x \in \mathcal M$
yields
\(
\xi(x)
=
G(q_x) v_x
+
R(q_x,v_x),
\)
where $G(q_x):=\nabla \xi(q_x)^\top\nabla\xi(q_x)\in \mathbb{R}^{m\times m}$ is the Gram matrix and the remainder satisfies
$ \|R(q_x,v_x)\| \le C \|v_x\|^2 $.
Conditional on $q_x$, the dominant term in the exponent is quadratic in $v_x$,
and the normal component is approximately Gaussian as $\varepsilon$ approaches to 0. The expected squared normal displacement is of order $\varepsilon$,
and the typical normal distance from the manifold is of order
$\sqrt{\varepsilon}$. To formalize this observation, Theorem \ref{thm:tubular_convergence} shows that the relaxed distribution family converges to the target probability measure $\mu$ supported on $\mathcal{M}$ in probability. 

\begin{theorem}[Weak Convergence of $\pi_\epsilon$]
\label{thm:tubular_convergence}
Let $\xi: \mathbb{R}^n \to \mathbb{R}^m$ be a constraint function satisfying Assumption \ref{ass:xi}, and let $V: \mathbb{R}^n \to \mathbb{R}$ be a potential function satisfying Assumption \ref{ass:potential}. Let $\pi_\varepsilon$ be defined by equation \eqref{eq:tubular_relaxation}. Then, as $\varepsilon \to 0$,
\begin{itemize}
    \item[(i)] There exists $C > 0$ such that
    \(
    \mathbb{E}_{\pi_\varepsilon}\left[\|v_x\|^2\right] \le C\varepsilon,
    \)
    where $v_x$ denotes the normal component of $x\in\mathcal{T}_\tau(\mathcal{M})$ under the decomposition $x = \Phi(q_x, v_x)$ with $q_x = P(x)$.
    
    \item[(ii)] The family of probability measures $\{\pi_\epsilon\}_{\epsilon>0}$ converge weakly to $\mu$ in \eqref{eq:target_measure} as $\epsilon \rightarrow 0$.
\end{itemize}
\end{theorem}
\begin{proof}
We first prove part (i). Let $\mathcal{T}_\tau(\mathcal{M})$ denote the global tubular neighborhood of $\mathcal{M}$, and let $\Phi$ be the normal bundle diffeomorphism given in Proposition \ref{prop:maximal_tube}, with inverse
\[
\Phi^{-1}(q,v) = q + \nabla \xi(q)\, v.
\]
Every $x \in \mathcal{T}_\tau(\mathcal{M})$ admits a unique decomposition $x = \Phi^{-1}(q_x,v_x)$, where $q_x = P(x) \in \mathcal{M}$ and $\nabla \xi (q_x)v_x \in N_{q_x}\mathcal{M}$.
Since $q_x \in \mathcal{M}$ satisfies $\xi(q_x)=0$, a Taylor expansion of $\xi$ at $q_x$ gives
\[
\xi(x)
=
\xi(q_x) + 
\nabla\xi(q_x)^\top\,(x-q_x)
+
R(q_x,v_x),
\]
where the remainder satisfies $\|R(q_x,v_x)\| \le C \|v_x\|^2$ for some constant $C>0$. Substituting $x-q_x = \nabla \xi(q_x)v_x$, we obtain
\[
\xi(x) = G(q_x) v_x + O(\|v_x\|^2).
\]
Thus, $\|\xi(x)\|^2 = v_x^\top G(q)^2 v_x + O(\|v_x\|^3)$. We first partition the integral for the normalizing constant $Z_\varepsilon$ defined by 
$$Z_\varepsilon := \int_{\mathcal{T}_\tau(\mathcal{M})} \exp\left(-V(x) - \frac{1}{\varepsilon}\|\xi(x)\|^2\right) dx + \int_{\mathcal{T}_\tau(\mathcal{M})^c} \exp\left(-V(x) - \frac{1}{\varepsilon}\|\xi(x)\|^2\right) dx.$$
On $\mathcal{T}_\tau(\mathcal{M})^c$, we have $\|\xi(x)\|^2 \ge \alpha > 0$. By Assumption \ref{ass:xi}(iii), $\int_{\mathcal{T}_\tau(\mathcal{M})^c} e^{-V(x)}dx < \infty$, so the exterior integral is bounded by $O(e^{-\alpha/\varepsilon})$. For the integral over $\mathcal{T}_\tau(\mathcal{M})$, we apply the coordinate transformation and the rescaling $v = \sqrt{\varepsilon}w$

$$\int_{\mathcal{M}} \int_{\|v\| < \tau} \exp\left(-V(q) - \frac{1}{\varepsilon}\left(v^\top G(q)^2 v + O(\|v\|^3)\right)\right) |J(q,v)| \, dv \, d\sigma_{\mathcal{M}}(q)$$
$$= \varepsilon^{m/2} \int_{\mathcal{M}} \int_{\|w\| < \tau/\sqrt{\varepsilon}} \exp\left(-V(q) - w^\top G(q)^2 w + O(\sqrt{\varepsilon}\|w\|^3)\right) |J(q,\sqrt{\varepsilon}w)| \, dw \, d\sigma_{\mathcal{M}}(q).$$
By Assumption \ref{ass:xi} and the Dominated Convergence Theorem, as $\varepsilon \to 0$, this converges to
$$\varepsilon^{m/2} \int_{\mathcal{M}} e^{-V(q)} J(q,0) \left( \int_{\mathbb{R}^m} e^{-w^\top G(q)^2 w} dw \right) d\sigma_{\mathcal{M}}(q).$$
The standard Gaussian integral over $\mathbb{R}^m$ gives $$\int e^{-w^\top G(q)^2 w} dw = \pi^{m/2} (\det G(q)^2)^{-1/2} = \pi^{m/2} (\det G(q))^{-1}.$$ 
Substituting $J(q,0) = \det G(q)^{1/2}$, we obtain the determinant factor
$$Z_\varepsilon = \varepsilon^{m/2} \pi^{m/2} \int_{\mathcal{M}} e^{-V(q)} (\det G(q))^{-1/2} \, d\sigma_{\mathcal{M}}(q) + O(e^{-\alpha/\varepsilon}).$$
Define $Z_0 = \int_{\mathcal{M}} e^{-V(q)} (\det G(q))^{-1/2} \, d\sigma_{\mathcal{M}}(q)$, we establish the asymptotic relationship $Z_\varepsilon \sim \varepsilon^{m/2} \pi^{m/2} Z_0$.

We can now prove the theorem part (i). We partition the expectation integral over $\mathcal{T}_\tau(\mathcal{M})$ and $\mathcal{T}_\tau(\mathcal{M})^c$
$$\mathbb{E}_{\pi_\varepsilon}[\|v_x\|^2] = \frac{1}{Z_\varepsilon} \int_{\mathcal{T}_\tau(\mathcal{M})} \|v_x\|^2 e^{-V(x) - \frac{1}{\varepsilon}\|\xi(x)\|^2} dx + \frac{1}{Z_\varepsilon} \int_{\mathcal{T}_\tau(\mathcal{M})^c} \|v_x\|^2 e^{-V(x) - \frac{1}{\varepsilon}\|\xi(x)\|^2} dx.$$
By Assumption \ref{ass:xi}(ii), the tail integral on $\mathcal{T}_\tau(\mathcal{M})^c$ decays exponentially like $O(e^{-\alpha/\varepsilon})$. Since $Z_\varepsilon \sim \varepsilon^{m/2}$, the ratio vanishes exponentially. For the integral over $\mathcal{T}_\tau(\mathcal{M})$, substituting $v = \sqrt{\varepsilon}w$ introduces a factor of $\varepsilon$ from $\|v\|^2 = \varepsilon\|w\|^2$. The remaining integral evaluates asymptotically to a bounded constant, which completes the proof for $\mathbb{E}_{\pi_\varepsilon}[\|v_x\|^2] \le C\varepsilon$.

It remains to prove the theorem part (ii). Let $B_R$ be a closed ball of radius $R$. We bound the exterior measure
$$\pi_\varepsilon(B_R^c) = \frac{1}{Z_\varepsilon} \int_{B_R^c \setminus \mathcal{T}_\tau(\mathcal{M})} e^{-V(x) - \frac{1}{\varepsilon}\|\xi(x)\|^2} dx + \frac{1}{Z_\varepsilon} \int_{B_R^c \cap \mathcal{T}_\tau(\mathcal{M})} e^{-V(x) - \frac{1}{\varepsilon}\|\xi(x)\|^2} dx.$$
The first term vanishes uniformly as $\varepsilon \to 0$ because the numerator decays exponentially $O(e^{-\alpha/\varepsilon})$ while the denominator decays like $\varepsilon^{m/2}$. For the second term inside the tubular neighborhood, the rescaling $v = \sqrt{\varepsilon}w$ extracts a factor of $\varepsilon^{m/2}$, which cancels the leading order of $Z_\varepsilon$. As $\varepsilon \to 0$, this yields
$$\lim_{\varepsilon \to 0} \pi_\varepsilon(B_R^c \cap \mathcal{T}_\tau(\mathcal{M})) = \frac{1}{Z_0} \int_{\mathcal{M} \setminus B_R} e^{-V(q)} (\det G(q))^{-1/2} \, d\sigma_{\mathcal{M}}(q).$$
By Assumptions \ref{ass:xi}(iii) and \ref{ass:potential}(ii), the integrand $e^{-V(q)} (\det G(q))^{-1/2}$ is globally integrable over $\mathcal{M}$. Thus, the tail integral vanishes as $R \to \infty$. This proves the tightness
$\lim_{R \to \infty} \sup_{\varepsilon > 0} \pi_\varepsilon(B_R^c) = 0.$
By Prokhorov's Theorem \citep[Theorem 16.16]{billingsley2013}, tightness guarantees relative compactness. To identify the specific limit $\mu$, let $\forall \varphi \in C_c(\mathbb{R}^n)$ be a continuous test function with compact support $K$. For sufficiently small $\varepsilon$, we evaluate
$$\lim_{\varepsilon \to 0} \int_{\mathbb{R}^n} \varphi(x) \pi_\varepsilon(dx) = \lim_{\varepsilon \to 0} \frac{1}{Z_\varepsilon} \int_{\mathcal{T}_\tau(\mathcal{M}) \cap K} \varphi(x) e^{-V(x) - \frac{1}{\varepsilon}\|\xi(x)\|^2} dx.$$
Applying the same normal coordinate transformation and dominated convergence, the numerator evaluates asymptotically to
$$\varepsilon^{m/2} \pi^{m/2} \int_{\mathcal{M}} \varphi(q) e^{-V(q)} (\det G(q))^{-1/2} \, d\sigma_{\mathcal{M}}(q).$$
Dividing by $Z_\varepsilon \sim \varepsilon^{m/2} \pi^{m/2} Z_0$, the scaling factors cancel entirely, yielding:$$\lim_{\varepsilon \to 0} \int_{\mathbb{R}^n} \varphi(x) \pi_\varepsilon(dx) = \int_{\mathcal{M}} \varphi(q) \mu(dq)$$This confirms the weak limit is precisely $\mu$.

\end{proof}

As a result of Theorem \ref{thm:tubular_convergence}, the potential function
restricted to the implicit manifold can be approximated through invertible local
projections from the normal bundle, provided the relaxation parameter
$\varepsilon$ is chosen sufficiently small. The relaxed probability measure $\pi_\varepsilon$ is supported on a
connected subset of the ambient space and is absolutely continuous with respect
to Lebesgue measure.
This makes it substantially easier to sample using standard MCMC methods such as
HMC Algorithm \ref{alg:HMC}.
Building on this property, a replica exchange scheme with appropriately designed
exchange moves between samples in $\mathcal{T}_\tau(\mathcal{M})$ 
and their projections onto
$\mathcal{M}$ can be implemented to explore the target measure $\mu$ supported on disconnected implicit
manifold.
\subsection{Replica Exchange}\label{sec:pt}

Replica exchange, also known as parallel tempering MCMC, is a classical
framework for sampling from multimodal probability distributions supported on Euclidean spaces \citep{J15}. The key idea is to
run multiple Markov chains at different temperatures: low-temperature chains exploit high-probability regions, while high-temperature chains facilitate global exploration by flattening the potential landscape.
Occasional exchange moves between chains allow information to be exchanged across
temperatures while preserving detailed balance with respect to the joint target
distribution.

Sampling from disconnected implicit manifolds shares a superficial similarity
with this setting. Disconnected components of the constraint manifold
$\mathcal{M}$ may be heuristically interpreted as distinct modes of the target
distribution, with each mode corresponding to a smooth connected submanifold.
However, this analogy is limited.
In our setting, the target measure $\mu$ in equation \eqref{eq:target_measure} is supported on the
implicit manifold $\mathcal{M}$ and admits a density only with respect to the
induced Hausdorff measure on $\mathcal{M}$. By contrast, any relaxed or tempered
distribution has a density with respect to the ambient $n$-dimensional Lebesgue
measure, such as the tubular relaxation
$\pi_\epsilon$ defined in equation \eqref{eq:tubular_relaxation}. Although
$\pi_\epsilon$ weakly converges to $\mu$ as $\epsilon$ converges to 0, the measures $\mu$ and $\pi_\epsilon$ are mutually
singular for
any $\epsilon>0$. Consequently, exchange moves between cold and hot chains cannot be
constructed using the classical replica exchange acceptance rule, which
assumes absolute continuity of all involved measures with respect to a common Lebesgue measure. A carefully designed exchange mechanism is therefore required
to preserve detailed balance while accounting for the
singular absolutely continuous properties between $\mu$ and $\pi_\epsilon$.

Let $\{q^{(k)}\}_{k\ge0}$ denote the cold Markov chain with state space
$\mathcal{M}$ and invariant measure $\mu$, and let $\{x^{(k)}\}_{k\ge0}$ denote the
hot Markov chain with state space $\mathbb{R}^n$ and invariant measure
$\pi_\epsilon$. By Propositions~\ref{prop:maximal_tube}, there exist a global tubular neighborhood $\mathcal{T}_\tau(\mathcal{M})$ such that the
 the normal
bundle decomposition $\Phi$ is a diffeomorphism onto its image.
Assume that, at a proposed exchange time, the chain states satisfy
\(
q^{(k)}\in\mathcal{M},\)
\(x^{(k)}\in\mathcal{T}_\tau(\mathcal{M}).
\)
Then there exists a unique decomposition for $x^{(k)}$ on its normal bundle 
\begin{equation}\label{eq:phi}
    \phi(q^{(k)},x^{(k)})=(q^{(k)}, q_k, v_k), \quad s.t. \quad x^{(k)} = \Phi^{-1}(q_k,v_k) , 
\end{equation} 
where $q_k$ is the normal projection of $x^{(k)}$ onto $\mathcal{M}$, and coordinate transformation $\phi$ is a diffeomorphism on $\mathcal{M}\times\mathcal{T}_\tau(\mathcal{M})$.
We now construct a deterministic exchange proposal that exchanges the manifold
components while retaining the normal coordinate. Define
\(
\widetilde S(q^{(k)}, q_k, v_k) = (q_k, q^{(k)}, v_k),
\)
and let the full exchange mapping be
\begin{equation}\label{eq:exachange}
    S := \phi^{-1}\circ \widetilde S\circ \phi :
(q^{(k)},x^{(k)}) \longmapsto (q_k, x_k),
\end{equation}
where
\begin{equation}\label{eq:phiinve} 
\begin{aligned} \phi^{-1}(q_k, q^{(k)}, v_k) = (q_k, x_k), \quad s.t. \quad x_k = \Phi^{-1}(q^{(k)},v_k). 
\end{aligned} 
\end{equation}
The cold chain state is replaced by the projected point of the hot
chain, while the hot chain state is reconstructed using the same normal
coordinate relative to the original cold chain position. If, in addition,
$x_k\in\mathcal{T}_\tau(\mathcal{M})$, then the mapping $S$ is an involution on
$\mathcal{M}\times\mathcal{T}_\tau(\mathcal{M})$
\begin{equation}\label{eq:invo}
    S^2(q^{(k)},x^{(k)}) = (q^{(k)},x^{(k)}),
\end{equation}
which follows directly from the uniqueness of the tubular decomposition.
It now remains to determine when such a proposal $(q_k,x_k)$ should be accepted
so that the joint Markov chain admits the correct distribution
\begin{equation}\label{eq:pi}
    \Pi(dq,dx) = \mu(dq)\,\pi_\epsilon(dx)
\quad \text{on } \mathcal{M}\times\mathbb{R}^n.
\end{equation}
A first observation is that since the hot chain $\{x^{(k)}\}_{k\ge0}$ has
stationary distribution $\pi_\epsilon$ supported on the full ambient space
$\mathbb{R}^n$, there is a positive probability that
$x^{(k)}\notin\mathcal{T}_\tau(\mathcal{M})$, in which case the tubular decomposition $\phi$ in equation \eqref{eq:phi} is
not well defined. Moreover, even when $x^{(k)}\in\mathcal{T}_\tau(\mathcal{M})$, the
reconstructed point $x_k$ defined by equation \eqref{eq:phiinve} may lie
outside $\mathcal{T}_\tau(\mathcal{M})$, in which case the involution property of $S$
fails.
To address these issues, we employ a \emph{reversible involution check}. Specifically, an exchange proposal is immediately rejected unless both the current state $x^{(k)}$ and the proposed state $x_k$ admit numerical projections onto $\mathcal{M}$, and the involution condition \eqref{eq:invo} is satisfied within a prescribed numerical tolerance. This procedure ensures that accepted proposals lie in the domain where the exchange map $S$ is well-defined and numerically involutive. The resulting proposal kernel preserves reversibility with respect to the target distribution and avoids regions where the projection or involution property may fail. A summary process for such a reversible check is presented in the follow Algorithm \ref{check}.

\begin{algorithm}[htbp]
\caption{Reversible involution check for exchange proposal}\label{check}
\begin{algorithmic}[1]
\Require Current state $(q^{(k)},x^{(k)})$, maximum nonlinear iteration $N$, tolerance $\eta$
\State Attempt to evaluate $\Phi(x^{(k)})=(q_k,v_k)$ by solving \eqref{eq:projection_system}
\If{the nonlinear solver fails to converge within $N$ iterations}
    \State \Return Current sate $(q^{(k)}, x^{(k)})$
\EndIf
\State Evaluate the reconstructed point 
\(
x_k = \Phi^{-1}(q^{(k)}, v_k)
\)
by
\(
x_k = q^{(k)}+\nabla \xi(q^{(k)}) v_k
\)
\State Attempt to evaluate $\Phi(x_k) = (q',v')$ by solving \eqref{eq:projection_system} again
\If{the nonlinear solver fails within $N$ iterations or $\Vert q^{(k)}-q'\Vert \ge \eta$}
    \State \Return Current sate $(q^{(k)}, x^{(k)})$
\EndIf
\State \Return $(q_k,x_k)$
\end{algorithmic}
\end{algorithm}
We define the admissible domain $\mathcal{D}\subset \mathcal{M}\times \mathbb{R}^n$ as the set of all joint states that pass the reversible involution check in Algorithm~\ref{check}. Satisfying the check condition does not guarantee that the joint state lies in the global tubular neighborhood $\mathcal{T}_\tau(\mathcal{M})$. The converse inclusion holds, namely $\mathcal{T}_\tau(\mathcal{M})\subset \mathcal{D}$.
This is because $\mathcal{D}$ may include additional states for which the projection and involution conditions hold numerically. This is sufficient for the validity of the algorithm, since the resulting numerical proposal kernel, combined with the Metropolis--Hastings acceptance criterion, preserves the correct target measure.

It is straightforward to observe that the exchange mapping $S$ in equation \eqref{eq:exachange} is not measure-preserving. To ensure consistency with the joint target measure $\Pi$ in equation \eqref{eq:pi}, the change of variables formula must be taken into account, which requires the Jacobian determinant of the mapping $S$. In the following Theorem \ref{thm:jacobian_gram}, we derive a general expression for the Jacobian of $S$ and introduce a computationally efficient approximation that enables practical evaluation of this Jacobian term.
\begin{theorem}[Jacobian determinant of $S$]
\label{thm:jacobian_gram}
Let 
\(
\mathcal M = \{q\in\mathbb R^n : \xi(q)=0\}
\)
be an implicit manifold whose constraint function satisfies Assumption \ref{ass:xi}. Define the Gram matrix
\(
G(q) := \nabla\xi(q)^\top \nabla\xi(q) \in \mathbb{R}^{m\times m}.
\)
Let $(q,x)\in\mathcal D$ be an admissible state with decomposition
\(
x = q_x + \nabla\xi(q_x)v,
\)
where
\(
\xi(q_x)=0.
\)
Let $U(q)\in\mathbb{R}^{n\times (n-m)}$ have orthonormal columns forming a basis of the tangent space $T_q\mathcal M$, satisfying
\(
U(q)^\top U(q)=I_{n-m}
\)
and
\(
\nabla\xi(q)^\top U(q)=0.
\) Then we have

\textnormal{(i)} The exact absolute value of the Jacobian determinant of the exchange map $S$ is
\begin{equation}\label{eq:jaco}
|J_S(q,x)| = \sqrt{\frac{|\det G(q)|}{|\det G(q_x)|}}
\cdot
\frac{
\left|\det\Big(I_{n-m} + \sum_{i=1}^m v_i U(q)^\top \nabla^2\xi_i(q) U(q)\Big)\right|
}{
\left|\det\Big(I_{n-m} + \sum_{i=1}^m v_i U(q_x)^\top \nabla^2\xi_i(q_x) U(q_x)\Big)\right|}.
\end{equation}

\textnormal{(ii)} Define the Gram determinant approximation
\begin{equation}\label{eq:appro}
|\widehat{J_S}(q,x)| :=
\sqrt{\frac{|\det G(q)|}{|\det G(q_x)|}}.
\end{equation}
Then there exists a constant $C>0$ such that the expected approximation error satisfies
\begin{equation}\label{eq:error}
\mathbb{E}_{\Pi}
\!\left(
\left\vert
|\widehat{J_S}(q,x)|
-
|J_S(q,x)|
\right\vert
\right)
\le C\sqrt{\epsilon}.
\end{equation}

\end{theorem}

\begin{proof}
We begin by deriving an explicit expression for the Jacobian determinant of the mapping $S$ in part (i). The exchange map can be written as the composition
\(
S = \phi^{-1} \circ \tilde S \circ \phi
\)
shown in equation \eqref{eq:exachange}, where the mappings are defined as follows
\begin{itemize}
\item $\phi(q,x) = (q,q_x,v)$ is the local coordinate mapping induced by the projection $\Phi$, with decomposition
\[
x = q_x + \nabla\xi(q_x)v,\qquad \xi(q_x)=0,
\]
\item $\tilde S(q,q_x,v) = (q_x,q,v)$ exchanges the two manifold points while leaving the normal coordinate $v$ unchanged,
\item $\phi^{-1}(q_x,q,v) = (q_x, q + \nabla\xi(q)v)$ reconstructs the original coordinates.
\end{itemize}
Applying the chain rule gives
\[
J_S(q,x)
=
J_{\phi^{-1}}(q_x,q,v)\,
J_{\tilde S}(q,q_x,v)\,
J_{\phi}(q,x).
\]
Since $\tilde S$ is a permutation of coordinates, its Jacobian determinant satisfies
\(
|\det J_{\tilde S}| = 1.
\)
Therefore,
\[
|J_S(q,x)|
=
\frac{
\left|\det J_{\phi^{-1}}(q_x,q,v)\right|
}{
\left|\det J_{\phi}(q,x)\right|
}.
\]
We now compute the Jacobian of $\phi^{-1}$. From equation \eqref{eq:dphi-1}, we have the Jacobian matrix of $\Phi^{-1}$ in block form
\[
J_{\Phi^{-1}}(q,v)
=
\begin{bmatrix}
I_n + \sum_{i=1}^m v_i \nabla^2\xi_i(q)
&
\nabla\xi(q)
\end{bmatrix},
\]
which is an $n\times n$ matrix.
To compute its determinant, we introduce an orthonormal basis adapted to the manifold. Let
\(
U(q)\in\mathbb R^{n\times (n-m)}
\)
have orthonormal columns spanning the cotangent space $T_q^*\mathcal M$, satisfying
\[
U(q)^\top U(q) = I_{n-m},\quad
\nabla\xi(q)^\top U(q) = 0.
\]
Define the orthogonal matrix
\(
C(q)
=
\begin{bmatrix}
U(q) & N(q)
\end{bmatrix},
\)
where $N(q) := \nabla\xi(q)\, G(q)^{-1/2}$ so that $C(q)$ is orthogonal and spans $\mathbb R^n$.
Since multiplication by an orthogonal matrix preserves determinants up to sign,
\[
|\det J_{\Phi^{-1}}(q,v)|
=
|\det(C(q)^\top J_{\Phi^{-1}}(q,v))|.
\]
Compute
\[
C(q)^\top J_{\Phi^{-1}}(q,v)
=
\begin{bmatrix}
U(q)^\top \left(I + \sum_{i=1}^m v_i \nabla^2\xi_i(q)\right) U(q)
&
U(q)^\top \nabla\xi(q)
\\
N(q)^\top \left(I + \sum_{i=1}^m v_i \nabla^2\xi_i(q)\right) U(q)
&
N(q)^\top \nabla\xi(q)
\end{bmatrix}.
\]
This matrix is block triangular ($U(q)^\top \nabla\xi(q)=0$), so its determinant equals the product of diagonal block determinants. Therefore,
\[
|\det J_{\Phi^{-1}}(q,v)|
=
\left|
\det\!\left(
I_{n-m}
+
\sum_{i=1}^m v_i U(q)^\top \nabla^2\xi_i(q) U(q)
\right)
\right|
\cdot
\sqrt{|\det G(q)|}.
\]
Applying this result to both $(q_x,v)$ and $(q,v)$ gives
\[
|J_S(q,x)|
=\sqrt{\frac{|\det G(q)|}{|\det G(q_x)|}}
\cdot
\frac{
\left|
\det\!\left(
I_{n-m}
+
\sum_{i=1}^m v_i U(q)^\top \nabla^2\xi_i(q) U(q)
\right)
\right|
}{
\left|
\det\!\left(
I_{n-m}
+
\sum_{i=1}^m v_i U(q_x)^\top \nabla^2\xi_i(q_x) U(q_x)
\right)
\right|
}.
\]
This establishes equation \eqref{eq:jaco} in part (i) of the theorem.

We now prove part (ii), which establishes the accuracy of the Gram determinant approximation.
Recall from part (i) that the exact Jacobian determinant of the exchange map satisfies
\begin{equation}\nonumber
|J_S(q,x)|
=
\sqrt{\frac{|\det G(q)|}{|\det G(q_x)|}}
\cdot
\frac{|\det A(q,v)|}{|\det A(q_x,v)|},
\end{equation}
where 
\begin{equation}\nonumber A(q,v) := I_{n-m} + \sum_{i=1}^m v_i\, U(q)^\top \nabla^2\xi_i(q) U(q). \end{equation}
Let $E(q,v) := \sum_{i=1}^m v_i\, U(q)^\top \nabla^2\xi_i(q) U(q)$. Under the assumption that the constraint Hessians $\nabla^2\xi_i$ are uniformly bounded on the manifold, there exists a constant $C > 0$ such that $\|E(q,v)\| \le C \|v\|$. Using the identity $\det(I + E) = 1 + \operatorname{tr}(E) + O(\|E\|^2)$, the ratio of the determinants in the normal bundle can be linearized as
\begin{equation}\nonumber
\frac{\det A(q,v)}{\det A(q_x,v)} = \frac{1 + \operatorname{tr}(E(q,v)) + O(\Vert v\Vert^2)}{1 + \operatorname{tr}(E(q_x,v)) + O(\Vert v\Vert^2)}  = 1 + \operatorname{tr}(E(q,v) - E(q_x,v)) + O(\Vert v\Vert^2).\end{equation}
Then the absolute error between the exact and approximate Jacobian is
\begin{equation}\nonumber
\left| |J_S(q,x)| - |\widehat{J_S}(q,x)| \right| = |\widehat{J_S}(q,x)| \cdot \left| \frac{\det A(q,v)}{\det A(q_x,v)} - 1 \right|.
\end{equation}
Substituting the previous linearization, we obtain
\begin{equation}\nonumber
\left| |J_S(q,x)| - |\widehat{J_S}(q,x)| \right| = |\widehat{J_S}(q,x)| \cdot \left| \sum_{i=1}^m v_i \Delta H_i(q, q_x) + O(\Vert v\Vert^2) \right|,
\end{equation}
where 
\begin{equation}\nonumber
    \Delta H_i(q, q_x) := \operatorname{tr}(U(q)^\top \nabla^2\xi_i(q) U(q)) - \operatorname{tr}(U(q_x)^\top \nabla^2\xi_i(q_x) U(q_x)).
\end{equation}
Taking the expectation with respect to $\Pi$ gives
\begin{equation}\nonumber
\mathbb{E}_{\Pi} [ | \sum_{i=1}^m v_i \Delta H_i + O(\Vert v\Vert^2) |] \le \sum_{i=1}^m \vert\Delta H_i\vert \cdot \mathbb{E}_{\pi_\varepsilon}[\Vert v_i\Vert] + O(\mathbb E_{\pi_\varepsilon}[\Vert v\Vert^2]).
\end{equation}
By Theorem \ref{thm:tubular_convergence}, we have 
\begin{equation}\nonumber
\mathbb{E}_{\Pi} \left[ \left| |J_S(q,x)| - |\widehat{J_S}(q,x)| \right| \right] = O(\sqrt{\epsilon}),
\end{equation}
completing the proof.
\end{proof}
 Geometrically, the exact volume change under the exchange map is determined by the combined deformation along the tangent and normal directions, as characterized by the determinant of the corresponding Jacobian.
The Gram determinant $|\det G(q)|$ measures the squared volume of the $m$-dimensional parallelepiped spanned by the constraint gradient vectors $\{\nabla\xi_i(q)\}_{i=1}^m$. This quantity quantifies the local volume scaling in the normal space and reflects the "thickness" of the tubular neighborhood around the manifold.
Consequently, the ratio
\(
\sqrt{|\det G(q)/\det G(q_x)|}
\)
captures the relative change in normal volume elements between the points $q$ and $q_x$ on the manifold. This variation dominates the overall volume change induced by the exchange map $S$ for small normal
displacements introduced by samples from $\pi_\epsilon$. This observation justifies the Gram determinant approximation and explains its accuracy in the small-noise regime.

With the analytical formula for the Jacobian determinant of the exchange map, a Metropolis--Hastings correction can be introduced to ensure that the resulting joint Markov chain admits the correct stationary distribution with respect to $\Pi$. Let $z^{(k)}:=(q^{(k)}, x^{(k)}) \in \mathcal{D}$ denote the current state in the admissible domain ensuring numerical involution, and let
\(
z_k:=(q_k, x_k) = S(q^{(k)}, x^{(k)})
\)
denote the proposal generated by the exchange map $S$. The Metropolis--Hastings acceptance probability is defined as
\begin{equation}\label{eq:accept}
\alpha\!\left( z^{(k)}, z_k \right)
=
\min\!\left\{
1,\,
\frac{\pi(z_k)}
     {\pi\!\left(z^{(k)}\right)}
\, |J_S\big(z^{(k)}\big)|
\right\},
\end{equation}
where $\pi(z)\,dz = \Pi(dz)$ denotes the density of the product measure $\Pi$.
The resulting exchange transition kernel $P_\epsilon$ is given by
\begin{equation}\label{eq:kernel}
    P_\epsilon\big(z, A\big)
=
\alpha\big(z, S(z)\big)\,
\mathbf{1}_{A}\!\big(S(z)\big)
+
\left[
1 - \alpha\big((z), S(z)\big)
\right]
\mathbf{1}_{A}(z),
\end{equation}
for any measurable set $A \subset \mathcal{D}$.
It will be shown in the next section that this transition kernel satisfies detailed balance with respect to $\Pi$, and is irreducible.
Furthermore, by replacing the exact Jacobian determinant \eqref{eq:jaco} with the Gram matrix approximation \eqref{eq:appro}, the convergence error is controlled by $\sqrt\epsilon$ in total variation distance.
We present the full replica exchange constrained HMC algorithm in Algorithm~\ref{alg:ptchmc}.
\begin{algorithm}[htbp]
\caption{Replica Exchange Constrained Hamiltonian Monte Carlo}
\label{alg:ptchmc}
\begin{algorithmic}[1]
\Require Target density $\mu$ on manifold $\mathcal{M}=\{q \in \mathbb{R}^n : \xi(q)=0\}$, temperature parameter $\epsilon>0$, exchange period $K\in\mathbb{N}$, step sizes $h_c,h_h>0$, trajectory lengths $L_c,L_h\in\mathbb{N}$, and total number of iterations $N$

\State Initialize cold-chain position $q^{(0)} \in \mathcal{M}$ and hot-chain position $x^{(0)} \in \mathbb{R}^n$

\For{$k = 0,1,\ldots,N-1$}

    \State Update cold chain: sample $q^{(k+1)}$ using Algorithm~\ref{alg:CHMC} with $h_c,L_c$
    
    \State Update hot chain: sample $x^{(k+1)}$ using Algorithm~\ref{alg:HMC} with $h_h,L_h$
    
    \If{$(k+1) \bmod K = 0$}
    
        \State Propose exchange
        \(
        (q', x') = S\big(q^{(k+1)}, x^{(k+1)}\big)
        \) using Algorithm \ref{check}
        
        \If{$\Vert (q',x') - (q^{(k+1)}, x^{(k+1)})\Vert\not=0$}
        
            \State Evaluate
            \(
            \left| J_S\big(q^{(k+1)}, x^{(k+1)}\big) \right|
            \)
            using equation \eqref{eq:jaco} or \eqref{eq:appro}
            
            \State Compute
            \(
            \alpha\!\left( (q^{(k+1)}, x^{(k+1)}), (q', x') \right)
            \)
            using equation \eqref{eq:accept}
            
            \State Sample $u \sim \mathrm{Uniform}(0,1)$
            
            \If{$u < \alpha$}
                \State
                \(
                (q^{(k+1)}, x^{(k+1)}) \leftarrow (q', x')
                \)
            \EndIf
            
        \EndIf
        
    \EndIf

\EndFor

\Return $\{(q^{(k)},x^{(k)})\}_{k=0}^{N}$
\end{algorithmic}
\end{algorithm}

\subsection{Convergence Properties}

We can now establish the detailed balance and irreducibility properties
of Algorithm~\ref{alg:ptchmc}, and prove convergence of the joint
Markov chain to the correct stationary distribution.
Recall that the cold chain targets the Hausdorff probability measure
$\mu$ supported on the implicit manifold
$\mathcal{M}$
while the hot chain targets the relaxed density $\pi_\epsilon$ supported on $\mathbb R^n$.
The joint target measure 
is  the product measure $\Pi$ in equation \eqref{eq:pi}.
Essentially the Algorithm~\ref{alg:ptchmc} consists of three components: a CHMC update preserving $\mu$, a standard HMC update preserving $\pi_\epsilon$, and a deterministic exchange proposal
mapping $S$ defined on the admissible domain $\mathcal D$,
with Jacobian $J_S(z)$ and
Metropolis--Hastings correction defined in equation \eqref{eq:accept}.

Existing results in the literature ensure that
CHMC is reversible, irreducible, and convergent
with respect to $\mu$ when $\mathcal M$ is connected and satisfies Assumption \ref{ass:xi} \citep{M3},
and that standard HMC has the same properties with respect
to $\pi_\epsilon$ for $\epsilon>0$.
The presence of multiple disconnected components of $\mathcal{M}$ prevents a single constrained chain from being globally irreducible. Therefore, our main goal is to demonstrate that the exchange mechanism restores global irreducibility of $\Pi$ within the admissible region $\mathcal{D}$, so that the cold chain $\{q^{(k)}\}_{k=0}^N$ explores different components and converges to the correct target density $\mu$. We impose the following assumption on the samplers.

\begin{assumption}\label{assm:samplers}
Let $P_\mu$ be the transition kernel of the CHMC algorithm~\ref{alg:CHMC} and $P_{\pi_\epsilon}$ be the transition kernel of the HMC algorithm~\ref{alg:HMC}. Suppose the implicit manifold $\mathcal{M}$ satisfies Assumption~\ref{ass:xi} and has $M$ connected components, i.e.\ $\mathcal{M} = \bigcup_{i=1}^{M} \mathcal{M}_i$. Then we assume:

(i) $P_\mu$ is reversible and $\mu$-irreducible on each connected component $\mathcal{M}_i$ for $i = 1, 2, \ldots, M$;

(ii) $P_{\pi_\epsilon}$ is reversible and $\pi_\epsilon$-irreducible on the ambient space $\mathbb{R}^n$.

\end{assumption}

Under Assumption~\ref{assm:samplers}, we construct the full transition kernel of Algorithm~\ref{alg:ptchmc}. Without loss of generality, suppose that an exchange step is performed after every $K$ iterations of both the CHMC and the HMC kernels. 
Since the cold and hot chains evolve independently between exchange steps, their joint evolution over $K$ steps is described by the product kernel
\[
\left(
P_\mu^K \otimes P_{\pi_\epsilon}^K
\right)
\big((q,x), dq' dx'\big)
=
P_\mu^K(q,dq') \,
P_{\pi_\epsilon}^K(x,dx'),
\]
where $P_\mu^K$ and $P_{\pi_\epsilon}^K$ denote $K$ successive applications of the constrained HMC kernel and the standard HMC kernel, respectively.
The full transition kernel acting on $\mathcal D$ is therefore defined as the composition
\(
P
:=
P_\epsilon
\circ
\left(
P_\mu^K \otimes P_{\pi_\epsilon}^K
\right),
\) where $P_\epsilon$ is the exchange kernel defined in equation \eqref{eq:kernel}.
We now state the main convergence result for Algorithm~\ref{alg:ptchmc}.

\begin{theorem}[Convergence]\label{thm:pt_full}
Let Assumption~\ref{assm:samplers} hold and let $P$ be the full transition kernel. Then the Markov chain generated by Algorithm~\ref{alg:ptchmc} satisfies

(i) The transition kernel $P$ is reversible with respect to $\Pi$ on $\mathcal{D}$.

(ii) The chain is $\Pi$-irreducible on $\mathcal{D}$.

(iii)For any initial state $z: = (q, x)\in\mathcal{D}$,
    \[
        \left\| P^n\left (z,\cdot \right ) - \Pi \right\|_{\mathrm{TV}} \;\longrightarrow\; 0 \quad \text{as } n\to\infty.
    \]
In particular, $\Pi$ is the unique stationary distribution of $P$ on $\mathcal{D}$, and the chain is ergodic.
\end{theorem}

\begin{proof}
We start by proving the reversibility of the kernel $P$.
Since $P_\mu$ is reversible with respect to $\mu$, for any measurable sets
$A_1,B_1 \subset \mathcal M$, we have
\[
\int_{A_1} \mu(dq)\, P^K_\mu(q,B_1)
=
\int_{B_1} \mu(dq)\, P^K_\mu(q,A_1).
\]
Similarly, since $P_{\pi_\epsilon}$ is reversible with respect to $\pi_\epsilon$, we have
\[
\int_{A_2} \pi_\epsilon(dx)\, P_{\pi_\epsilon}^K(x,B_2)
=
\int_{B_2} \pi_\epsilon(dx)\, P_{\pi_\epsilon}^K(x,A_2)
\]
for any measurable $A_2,B_2 \subset \mathbb R^d$.
Consider now the product kernel acting on $(q,x)\in\mathcal{M}\times \mathbb R^n$,
\[
\left(
P_\mu^K \otimes P_{\pi_\epsilon}^K
\right)
\big((q,x), (dq',dx')\big)
=
P_\mu^K(q,dq')\,
P_{\pi_\epsilon}^K(x,dx').
\]
Let $A,B \subset \mathcal D$ be measurable,
we obtain
\[
\int_A \Pi(dq,dx)\,
\left(
P_\mu^K \otimes P_{\pi_\epsilon}^K
\right)
\big((q,x),B\big)
=
\int_A
\mu(dq)\,\pi_\epsilon(dx)
\int_B
P_\mu^K(q,dq')\,
P_{\pi_\epsilon}^K(x,dx').
\]
By Fubini's theorem and the separate reversibility of
$P_\mu^K$ and $P_{\pi_\epsilon}^K$, we may exchange
$(q,x)$ and $(q',x')$, so that the product kernel
$P_\mu^K \otimes P_{\pi_\epsilon}^K$
is reversible with respect to $\Pi$ on $\mathcal{M}\times \mathbb R^n$.
By construction, the exchange kernel $P_\epsilon$ in equation \eqref{eq:kernel}
is a Metropolis--Hastings kernel with acceptance probability
$\alpha$ defined in eqation \eqref{eq:accept}. 
Therefore, for any measurable $A,B\subset\mathcal D$,
\[
\int_A \Pi\big(dq,dx\big)\,P_\epsilon\big ((q,x),B)
=
\int_B \Pi(dq,dx)\,P_\epsilon\big((q,x),A\big),
\]
and hence $P_\epsilon$ is reversible with respect to $\Pi$ on $\mathcal{D}$.
Since both
\(
P_\mu^K \otimes P_{\pi_\epsilon}^K\) and 
\(P_\epsilon
\)
are reversible with respect to the same measure $\Pi$,
their composition $P$
is also reversible with respect to $\Pi$.

We next show that the Markov chain generated by Algorithm~\ref{alg:ptchmc}
is $\Pi$-irreducible on $\mathcal D$.
Let $z=(q,x)\in\mathcal D$ and let $A\subset\mathcal D$ be measurable with
\(
\Pi(A) > 0.
\)
Since
\(
\Pi(dq,dx)=\mu(dq)\,\pi_\epsilon(dx),
\)
Fubini's theorem implies that there exist measurable sets
$A_q\subset\mathcal M$ and $A_x\subset\mathbb R^d$
such that
\[
\mu(A_q)>0,
\qquad
\pi_\epsilon(A_x)>0,
\qquad
A_q\times A_x \subset A.
\]
By Assumption~\ref{assm:samplers}, the kernel $P_\mu$ is
$\mu$-irreducible on $\mathcal M$ and
$P_{\pi_\epsilon}$ is $\pi_\epsilon$-irreducible on $\mathbb R^d$.
Hence their $K$-step kernels $P_\mu^K$ and
$P_{\pi_\epsilon}^K$ are also irreducible with respect to
$\mu$ and $\pi_\epsilon$, respectively.
For the product kernel
\(
P_\mu^K \otimes P_{\pi_\epsilon}^K\),
irreducibility of the marginals yields
\[
P_\mu^K(q,A_q) > 0,
\qquad
P_{\pi_\epsilon}^K(x,A_x) > 0.
\]
Therefore,
\[
\left(
P_\mu^K \otimes P_{\pi_\epsilon}^K
\right)
\big((q,x),A_q\times A_x\big)
=
P_\mu^K(q,A_q)\,
P_{\pi_\epsilon}^K(x,A_x)
>0.
\]
Since the full kernel is
\(
P
=
P_\epsilon
\circ
\left(
P_\mu^K \otimes P_{\pi_\epsilon}^K
\right),
\)
and the exchange kernel $P_\epsilon$ has strictly positive
acceptance probability on $\mathcal D$,
it follows that $P$ preserves positivity of reachable sets.
Consequently, there exists an integer $n\ge1$ such that
\[
P^n\big((q,x),A_q\times A_x\big) > 0,
\]
which gives the irreducibility of $P$ on $\mathcal{D}$.

It remains to show the convergence. Since $P$ is reversible with respect to $\Pi$, it follows that $\Pi$ is stationary distribution of the chain generated by Algorithm \ref{alg:ptchmc}.
Since $P$ is irreducible and satisfies
\[
P(z,\{z\}) \ge 1-\alpha(z,S(z)) > 0
\]
for $\forall z = (q,x)\in\mathcal{D}$,
the chain is aperiodic.
Therefore, for any initial state $z \in \mathcal D$,
\[
\left\|
P^n(z,\cdot) - \Pi
\right\|_{\mathrm{TV}}
\longrightarrow 0.
\]
Uniqueness of the stationary distribution follows from irreducibility, and ergodicity follows from convergence in total variation.
\end{proof}
With convergence of the Markov chain established in Theorem~\ref{thm:pt_full}, we now consider the effect of replacing the exact Jacobian determinant $|J_S(z)|$ in the exchange acceptance probability with the Gram approximation $\widehat{J}_S(z)$ defined in equation \eqref{eq:appro} for arbitrary $z=(q,x)\in\mathcal{D}$. This modification results in a perturbed Metropolis--Hastings acceptance probability
\[
\widehat{\alpha}(z,S(z))
=
\min\!\left\{
1,\,
\frac{\pi(S(z))}{\pi(z)}
\,|\widehat{J}_S(z)|
\right\},
\]
and defines a corresponding approximate exchange kernel $\widehat{P}_\epsilon$. The full approximate transition kernel is therefore
\(
\widehat{P}
=
\widehat{P}_\epsilon
\circ
(P_{\pi_\epsilon}^K
\otimes
P_\mu^K).
\)
Since the Metropolis--Hastings construction guarantees detailed balance with respect to its invariant distribution, the approximate kernel $\widehat{P}_\epsilon$ admits a stationary distribution $\widehat{\Pi}$ satisfying
\[
\widehat{\Pi}(dz)\,\widehat{P}_\epsilon(z,dz')
=
\widehat{\Pi}(dz')\,\widehat{P}_\epsilon(z',dz).
\]
Consequently, the full approximate kernel $\widehat{P}$ admits $\widehat{\Pi}$ as its stationary distribution, and by the same arguments as in Theorem~\ref{thm:pt_full}, the Markov chain generated using the Gram approximation satisfies
\[
\left\|
\widehat{P}^n(z,\cdot)
-
\widehat{\Pi}
\right\|_{\mathrm{TV}}
\longrightarrow 0,
\qquad
\forall z \in \mathcal D.
\]
It remains to quantify the difference between the approximate invariant distribution $\widehat{\Pi}$ and the true target distribution $\Pi$, which we present the following Theorem \ref{thm:gram_bias}. 

\begin{theorem}[Bias of the Gram approximation]
\label{thm:gram_bias}
Let $\mathcal{M} = \{q \in \mathbb{R}^n : \xi(q) = 0\}$ be an implicit manifold 
satisfying Assumption~\ref{ass:xi}. Assume the admissible domain $\mathcal{D}\subset \mathcal{M}\times \mathbb R^n$ is compact. Let $P$ denote the exact transition kernel with invariant distribution $\Pi$, and let $\widehat{P}$ denote the perturbed kernel obtained by replacing $|J_S(z)|$ with its Gram approximation $|\widehat{J}_S(z)|$. Then there exists a constant $C > 0$ such that
\[
\|\widehat{\Pi} - \Pi\|_{\mathrm{TV}} \le C \sqrt{\varepsilon},
\]
where $\widehat{\Pi}$ is the invariant distribution of $\widehat{P}$.
\end{theorem}

\begin{proof}
The two kernels share the same deterministic proposal $z' = S(z)$ and differ 
only in their acceptance probabilities. Specifically,
\[
P(z,\cdot) 
= \alpha(z)\,\delta_{S(z)}(\cdot) + (1-\alpha(z))\,\delta_z(\cdot), \qquad
\widehat{P}(z,\cdot) 
= \widehat{\alpha}(z)\,\delta_{S(z)}(\cdot) + (1-\widehat{\alpha}(z))\,\delta_z(\cdot),
\]
where
\[
\alpha(z) 
= \min\!\left\{1,\,\frac{\pi(S(z))}{\pi(z)}\,|J_S(z)|\right\}, 
\qquad 
\widehat{\alpha}(z) 
= \min\!\left\{1,\,\frac{\pi(S(z))}{\pi(z)}\,|\widehat{J}_S(z)|\right\}.
\]
Therefore
\[
\|\widehat{P}(z,\cdot) - P(z,\cdot)\|_{\mathrm{TV}}
= |\widehat{\alpha}(z) - \alpha(z)|.
\]
Setting $\delta(z) := |J_S(z)| - |\widehat{J}_S(z)|$ and using the fact that 
$a \mapsto \min\{1,a\}$ is $1$-Lipschitz on $(0,\infty)$,
\[
|\widehat{\alpha}(z) - \alpha(z)|
\le 
\frac{\pi(S(z))}{\pi(z)}\,|\delta(z)|.
\]
Hence
\[
\|\widehat{P}(z,\cdot) - P(z,\cdot)\|_{\mathrm{TV}}
\le 
\frac{\pi(S(z))}{\pi(z)}\,|\delta(z)|. 
\]
We take the expectation of the above formula with respect to the invariant measure $\Pi(dz) = \pi(z)\,dz$
\[
\mathbb{E}_{\Pi}
\!\left[
\|\widehat{P}(z,\cdot) - P(z,\cdot)\|_{\mathrm{TV}}
\right]
\le 
\int_{\mathcal{D}} 
\frac{\pi(S(z))}{\pi(z)}\,|\delta(z)|\,\pi(z)\,dz.
\]
By assumption, Algorithm~\ref{check} restricts the state space to a compact domain $\mathcal{D}$. Since the target density $\pi(z)$ is continuous and strictly positive on $\mathcal{D}$, the ratio of the densities is uniformly bounded. Thus, there exists a constant $M > 0$ such that $\sup_{z \in \mathcal{D}} \frac{\pi(S(z))}{\pi(z)} \le M$. 
Applying this bound, we obtain
\[
\int_{\mathcal{D}} \frac{\pi(S(z))}{\pi(z)}\,|\delta(z)|\,\pi(z)\,dz 
\le 
M \int_{\mathcal{D}} |\delta(z)|\,\pi(z)\,dz 
= M\,\mathbb{E}_{\Pi}[|\delta|].
\]
By Theorem~\ref{thm:jacobian_gram}(ii), the expected absolute error of the Gram approximation satisfies $\mathbb{E}_{\Pi}[|\delta|] \le C' \sqrt{\varepsilon}$. Hence, the expected total variation difference between the kernels is bounded by:
\[
\mathbb{E}_{\Pi}
\!\left[
\|\widehat{P}(z,\cdot) - P(z,\cdot)\|_{\mathrm{TV}}
\right]
\le 
M C' \sqrt{\varepsilon}.
\]
Because $P$ is reversible and irreducible with respect to $\Pi$ by Theorem~\ref{thm:pt_full}(i)--(ii), and the effective state space $\mathcal{D}$ is restricted to a compact valid tube by the algorithm, the Doeblin condition holds on $\mathcal{D}$. Thus, $P$ is uniformly ergodic. 
For uniformly ergodic Markov chains, $L^1$-based perturbation bounds \citep{J30} guarantee that the distance between the invariant distributions is bounded by the integrated difference of the transition kernels. Therefore, there exists a constant $C_{\mathrm{erg}} > 0$ such that
\[
\|\widehat{\Pi} - \Pi\|_{\mathrm{TV}}
\le C_{\mathrm{erg}}\, \mathbb{E}_{\Pi}\!\left[ \|\widehat{P}(z,\cdot) - P(z,\cdot)\|_{\mathrm{TV}} \right]
\le C \sqrt{\varepsilon},
\]
where $C = C_{\mathrm{erg}} M C'$, which completes the proof.
\end{proof}

From Theorem \ref{thm:gram_bias}, we know that replacing the exact Jacobian $J_S$ in equation \eqref{eq:jaco} with the Gram approximation $\widehat{J}_S$ in equation \eqref{eq:appro} introduces an average bias of order $\sqrt{\epsilon}$. Consequently, the relaxation parameter must be chosen sufficiently small to minimize the approximation error with respect to the target distribution $\Pi$. In addition, selecting a small relaxation parameter $\epsilon$ increases the probability that a proposal passes the involutive checks in Algorithm \ref{check}. It increases the probability that a proposed state $z$ lies in the accessible domain $\mathcal{D}$. This domain typically concentrates near the manifold $\mathcal{M}$, particularly when the manifold has high curvature. Importantly, our theoretical guarantees rely on the assumption that $\mathcal{D}$ is compact. This condition is generally satisfied as the involution check in Algorithm \ref{check} restricts the working domain by rejecting any proposals that fail to converge within finite iterative steps. To strictly guarantee the compactness of $\mathcal{D}$, one can impose an additional restriction that rejects proposals where the magnitude of the normal displacement $\Vert v \Vert$ exceeds a predefined threshold. This ensures that the auxiliary variables remain strictly bounded.

As discussed in Section \ref{sec:pt}, the presence of disconnected components in the manifold is heuristically analogous to multimodality in probability distributions. When $\epsilon$ is chosen to be small in order to increase the probability of proposing samples within $\mathcal{D}$ and to reduce the bias in the stationary distribution, the hot chain may have a very low probability of transitioning between different branches of the manifold. As a result, although the Markov chain remains theoretically ergodic, transitions between disconnected components may require prohibitively large steps.
To address this issue, one can use a sequence of samplers associated with increasing relaxation parameters $0<\epsilon_1<\cdots<\epsilon_I$. The chain with smallest relaxation parameter corresponds to the chain that most closely approximates the target distribution on the manifold, while chains with larger relaxation parameters explore progressively larger neighborhoods of $\mathcal{M}$. Exchange move defined through the mapping $S$ in equation \eqref{eq:exachange} between constrained sampler on $\mathcal{M}$ and unconstrained sampler on $\mathbb R^n$ is performed only for the coldest chain (the one with relaxation parameter $\epsilon_1$), while exchanges among the remaining chains follow the standard parallel tempering HMC (PT-HMC) framework as discussed in \citep{J15, J23, j24}.
In this way, chains with larger relaxation parameters can explore the ambient space more freely and traverse between different branches of the manifold. Through the exchange mechanism, this global exploration is gradually propagated to the constrained chain, which ultimately samples from the target measure $\mu$ on $\mathcal{M}$ even when the manifold contains disconnected components.

\section{Case Studies}\label{sec5}

We present three case studies to demonstrate the effectiveness of Algorithm~\ref{alg:ptchmc} for sampling from probability measures supported on disconnected manifolds.
The first example is a synthetic benchmark designed to illustrate the ability of the proposed method to transition between disconnected components. We also compare the exact Metropolis--Hastings correction in equation \eqref{eq:accept}, based on the exact Jacobian determinant \eqref{eq:jaco}, with its Gram approximation \eqref{eq:appro} to assess the bias introduced by the approximation.
The second example arises from infectious disease modeling, where structural non-identifiability leads to non-unique parameter estimates. In this setting, multiple parameter values yield identical model outputs, forming a manifold in parameter space that may decompose into disconnected components. 
The third example is drawn from molecular dynamics. We consider a tetrahedral molecule with holonomic constraints enforcing fixed bond lengths and angles, resulting in a low-dimensional manifold. Due to symmetry, the manifold consists of two disconnected components corresponding to enantiomeric configurations. We evaluate the ability of the proposed algorithm to explore both components and recover the target distribution.

\subsection{Synthetic Example}\label{C1}
We first consider a synthetic example in $\mathbb{R}^2$ consisting of four disconnected ellipses. The manifold is defined as
\(
\mathcal{M} = \bigcup_{k=1}^{4} \mathcal{M}_k,
\)
where each component $\mathcal{M}_k$ is an ellipse given by
\[
\mathcal{M}_k = \left\{ q \in \mathbb{R}^2 : 
\frac{(q_1 - c_{k,1})^2}{a_k^2} + \frac{(q_2 - c_{k,2})^2}{b_k^2} - 1 = 0 \right\},
\]
with centers $c_k = (c_{k,1}, c_{k,2}) \in \mathbb{R}^2$ and semi-axis parameters 
$a = (a_1,\dots,a_4)$ and $b = (b_1,\dots,b_4)$.
To express the manifold in the implicit form required by the constrained Hamiltonian framework, we define the constraint function
\[
\xi(q) = \prod_{k=1}^{4} \left( 
\frac{(q_1 - c_{k,1})^2}{a_k^2} + \frac{(q_2 - c_{k,2})^2}{b_k^2} - 1 
\right).
\]
The manifold can then be written as the zero level set
\(
\mathcal{M} = \{ q \in \mathbb{R}^2 : \xi(q) = 0 \}.
\)
Since each component $\mathcal{M}_k$ is smooth and compact, and the components are disjoint, the manifold $\mathcal{M}$ satisfies Assumption~\ref{ass:xi}.
We let the target distribution $\mu$ to be the uniform probability measure on $\mathcal{M}$, i.e.,
\(
\mu(dq) \propto \delta(\xi(q)) \, dq,
\)
where $\delta(\cdot)$ denotes the Dirac measure supported on the constraint set $\mathcal{M}$.
Sampling from this distribution is challenging for standard constrained samplers due to the disconnected structure of $\mathcal{M}$. In particular, once a Markov chain is initialized on a given component, standard constrained Hamiltonian Monte Carlo (CHMC) methods, such as Algorithm~\ref{alg:CHMC}, remain confined to that component and cannot transition between distinct ellipses.

In contrast, we apply Algorithm~\ref{alg:ptchmc} with a fixed relaxation parameter $\epsilon = 0.3$. The method employs two coupled chains: a constrained chain sampling on the manifold, and a relaxed chain targeting a smoothed distribution in the ambient space, defined by $\pi_\epsilon(dx) \propto \exp\left(-\frac{1}{\epsilon}|\xi(x)|^2\right) dx$. The relaxed chain explores a neighborhood of $\mathcal{M}$ in $\mathbb{R}^2$, thereby enabling transitions between distinct components.
With the replica exchange framework discussed in Section~\ref{sec:pt}, sampler states are periodically exchanged between the constrained and relaxed chains. This mechanism facilitates the transfer of samples across different components and generates uniform samples across all components of $\mathcal{M}$.

For the experimental setup, we consider four centers corresponding to $\mathcal{M}_1$ to $\mathcal{M}_4$, located at $(2,0)$, $(-2,0)$, $(0,2)$, and $(0,-2)$, respectively. The associated semi-axis parameters are $a = (1.6, 0.8, 1.2, 0.6)$ and $b = (0.6, 1.4, 0.9, 1.3)$. The objective is to sample from the target probability measure $\mu$ defined on this manifold.
To compare the traditional CHMC method with the proposed approach, we run both algorithms under identical settings with the same initialization on $\mathcal{M}_1$. For the proposed Algorithm~\ref{alg:ptchmc}, the relaxed chain is initialized at points in the ambient space near $\mathcal{M}_1$.
The results are presented in Figure~\ref{fig1}. In Figure~\ref{fig1a}, the standard CHMC sampler is observed to remain confined to a single connected component determined by its initialization. In contrast, Figure~\ref{fig1b} shows that the proposed method successfully explores all four components through replica exchange. Figure~\ref{fig1c} displays samples from the relaxed chain in the ambient space. Finally, the trace plot of the constrained chain in Figure~\ref{fig1d} demonstrates frequent transitions across all four ellipses, indicating effective mixing over the disconnected manifold.

\begin{figure}[H]
\centering
\begin{subfigure}[b]{0.45\textwidth}
    \centering
    \includegraphics[scale=0.8]{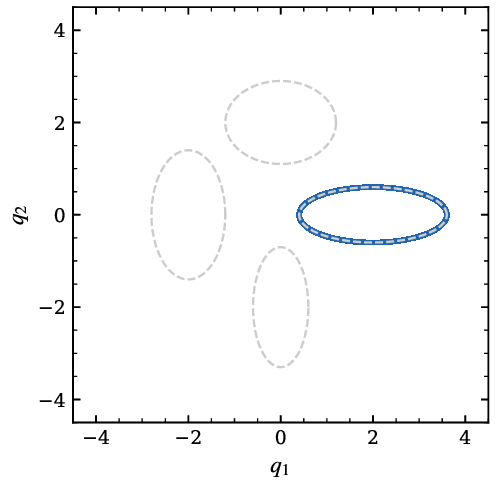}
    \caption{Algorithm \ref{alg:CHMC} result}
    \label{fig1a}
\end{subfigure}%
\hfill
\begin{subfigure}[b]{0.45\textwidth}
    \centering
    \includegraphics[scale=0.8]{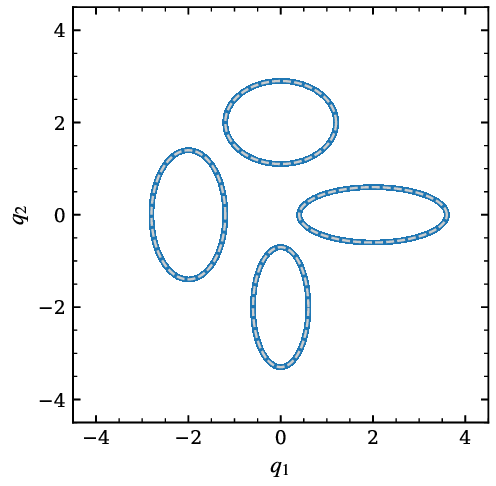}
    \caption{Algorithm \ref{alg:ptchmc} result}
    \label{fig1b}
\end{subfigure}

\vspace{0.3cm}

\begin{subfigure}[b]{0.45\textwidth}
    \centering
    \includegraphics[scale=0.8]{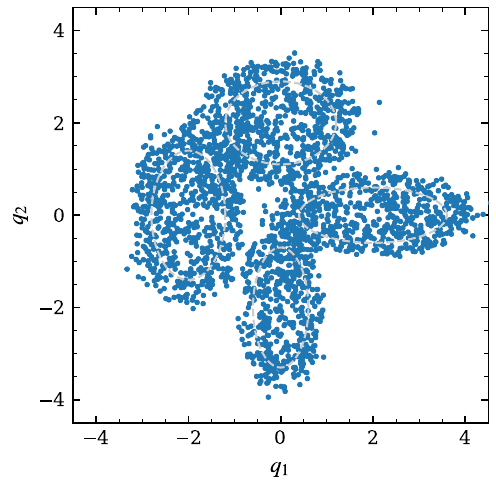}
    \caption{Relaxed chain samples}
    \label{fig1c}
\end{subfigure}%
\hfill
\begin{subfigure}[b]{0.45\textwidth}
    \centering
    \includegraphics[scale=0.8]{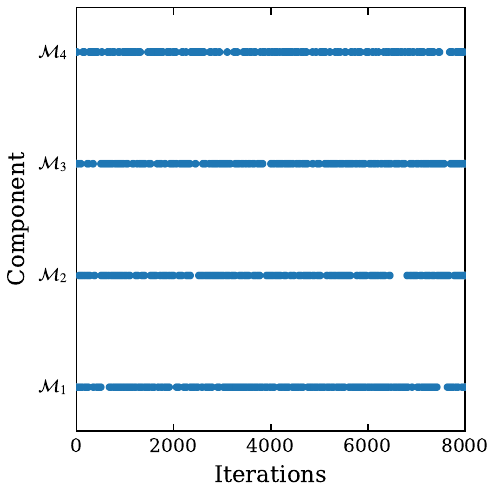}
    \caption{Constrained chain trace plot}
    \label{fig1d}
\end{subfigure}

\caption{Comparison of Algorithm~\ref{alg:CHMC} and Algorithm~\ref{alg:ptchmc} on a synthetic example with four disconnected elliptical components.}

\label{fig1}
\end{figure}

Having demonstrated that Algorithm~\ref{alg:ptchmc} enables efficient exploration across disconnected components, we now examine the bias introduced by replacing the exact Jacobian determinant in the Metropolis--Hastings acceptance probability with its Gram approximation in equation \eqref{eq:appro}, as analyzed in Theorem \ref{thm:gram_bias}.
In this synthetic example, both the exact Jacobian and its Gram approximation admit closed-form expressions. This allows us to directly quantify the approximation error and verify the theoretical result that the induced total variation distance between the sampled posterior and the target distribution is of order \(\sqrt{\epsilon}\).
To illustrate this behavior numerically, we run Algorithm~\ref{alg:ptchmc} over a sequence of relaxation parameters
\(
\epsilon_i = 0.1 + (i-1)\times 0.1,  i = 1,2,\dots,10.
\)
For each \(\epsilon\), we estimate the total variation error between the empirical distribution and the target. According to Theorem~\ref{thm:gram_bias}, the expected magnitude of the bias in total variation induced by the Gram approximation satisfies
\(
\mathrm{TV}\bigl(\widehat{\Pi}, \Pi\bigr) = O(\sqrt{\epsilon}).
\)
The estimated total variation error and posterior discrepancy are reported in Figure~\ref{fig2b}. The result is consistent with the theoretical predictions. Furthermore, by examining the convergence behavior through the sample mean of the \(q_1\) coordinate (Figure~\ref{fig2c}), we observe that the chains targeting \(\widehat{\Pi}\) and \(\Pi\) converge at similar rates, with only a small bias in their converged values.

\begin{figure}[h]
\centering
\begin{subfigure}[b]{0.45\textwidth}
    \includegraphics[width=\linewidth]{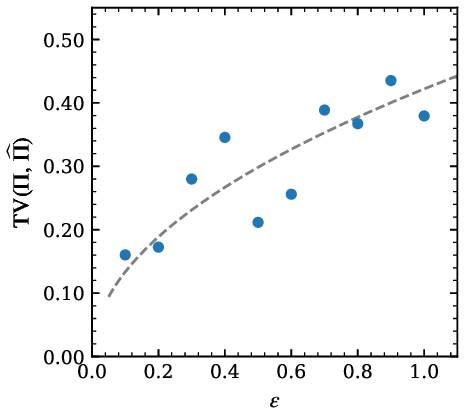}
    \caption{Total variation error}
    \label{fig2b}
\end{subfigure}%
\hfill
\begin{subfigure}[b]{0.45\textwidth}
    \includegraphics[width=\linewidth]{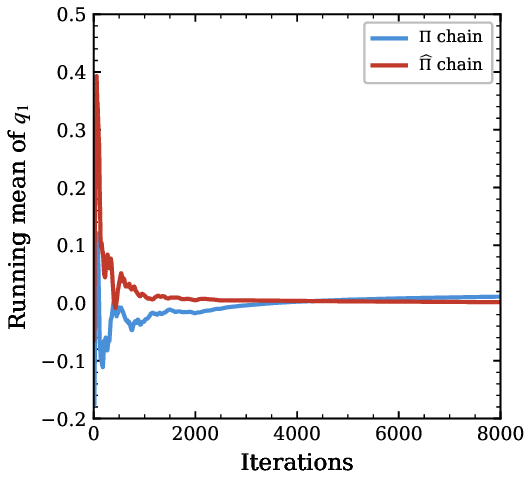}
    \caption{Convergence comparison}
    \label{fig2c}
\end{subfigure}

\caption{Numerical validation of Theorem~\ref{thm:gram_bias} using exact and Gram-approximated Jacobian determinants in Algorithm~\ref{alg:ptchmc}.}
\label{Fig2}
\end{figure}

\subsection{Infectious Disease Models}\label{C2}

The next numerical example arises from Bayesian inverse problems in infectious disease modeling \citep{J3, J8}. In this setting, one begins with a specified and evaluable forward model, typically derived from a system of differential equations describing the dynamics of disease transmission. The goal is to infer unknown model parameters from partially observed model outputs.
A fundamental challenge in such inverse problems is non-identifiability: multiple parameter values may produce equivalent model outputs, making exact parameter recovery impossible \citep{J14}. A common approach to address this issue is to characterize the non-identifiable structure of the model by identifying combinations of parameters that are structurally identifiable. This leads to the construction of a manifold of equivalent parameter sets, each yielding the same observable behavior. 
In this example, we show that a simple SIR epidemic model with two infectious compartments exhibits non-identifiability, and show how the proposed algorithm can be used to help explore the non-identifable manifolds.

Consider an infectious disease model with state variables 
$(S(t), I_1(t), I_2(t), R(t)) \in \mathbb{R}^4$, representing the susceptible population, two infectious populations corresponding to distinct viral strains, and the recovered population, respectively. We denote the model parameters by the vector
\(
\theta := (\beta_1, \beta_2, \gamma_1, \gamma_2, \rho) \in B \subset \mathbb{R}^5,
\)
where $\beta_1$ and $\beta_2$ are the transmission rates associated with the two strains, $\gamma_1$ and $\gamma_2$ are the corresponding recovery rates, and $\rho$ is the laboratory confirmed case reporting rate.
The model assumes that the observable quantity is reported incidence, given by a fraction $\rho$ of the total infectious population, i.e.,
\(
h(t,\theta): = \rho \big(I_1(t) + I_2(t)\big),
\)
which is a standard observation mechanism in epidemiological models with multiple circulating strains, such as influenza \citep{J27}. Figure~\ref{fig:SIR_model} presents both the underlying differential equations and a representative simulation of the model dynamics. The primary objective is to infer plausible parameter values from an observed case trajectory to enable quantitative analysis and informing public health decision-making. However, the input--output structure of this model induces structural non-identifiability. 
\begin{figure}[htbp]
    \centering
    % Left: equations
    \begin{minipage}[c]{0.45\linewidth}
        \[
        \begin{aligned}
        &S' = -\beta_1 S I_1 - \beta_2 S I_2 \\
        &I_1' = \beta_1 S I_1 - \gamma_1 I_1 - \rho I_1 \\
        &I_2' = \beta_2 S I_2 - \gamma_2 I_2 - \rho I_2 \\
        &R' = \gamma_1 I_1 + \gamma_2 I_2 \\
        &h(t,\theta) = \rho (I_1 + I_2) \quad \text{(output function)}
        \end{aligned}
        \]
    \end{minipage}%
    % Right: figure
    \begin{minipage}[c]{0.45\linewidth}
        \centering
        \includegraphics[width=\linewidth]{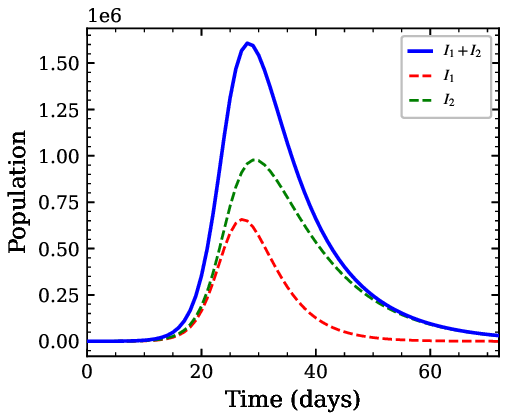}
    \end{minipage}
    \caption{SIR model with two infectious compartments. Left: governing differential equations. Right: simulation results.}
    \label{fig:SIR_model}
\end{figure}
Intuitively, the symmetry between the parameters of the two viral strains implies that the model is non-identifiable. Consequently, the full parameter vector $\theta$ cannot be uniquely identified from observations of $h(t,\theta)$ alone.
Structural identifiability analysis reveals that only certain combinations of parameters are identifiable \citep{j25,j26}. In this case, the identifiable quantities are characterized by the mapping $\xi : \mathbb{R}^5 \to \mathbb{R}^4$ defined by
\[
\xi(\theta) = 
\left(
\gamma_1 + \gamma_2,\;
\gamma_1 \gamma_2,\;
\frac{\beta_1 + \beta_2}{\rho},\;
\frac{\rho(\gamma_1 - \gamma_2)}{\beta_1 - \beta_2}
\right).
\]
This mapping reduces the five-dimensional parameter space to a four-dimensional space of identifiable combinations. As a result, $\xi$ is not injective, and each level set
\(
\mathcal{M}_{\vartheta} := \bigl\{\theta \in B : \xi(\theta) = \vartheta\bigr\}
\)
defines a one-dimensional algebraic manifold consisting of all parameter configurations that yield identical model outputs. This manifold characterizes the structural non-identifiability inherent in the model shown in Figure~\ref{fig:SIR_model}.
Suppose that we observe a confirmed case trajectory $y$, consider the Bayesian inverse problem
\[
\mu(\theta \mid y) 
\propto \exp\bigl(-\varrho\bigl(y - h(\cdot,\theta)\bigr)\bigr)\,\mu_0(\theta),
\]
where $\mu_0$ denotes a prior distribution. Since the likelihood depends on $\theta$ only through $\xi(\theta)$, it follows that the posterior distribution, up to the prior, is constant along each level set $\mathcal{M}_{\vartheta}$.
Therefore, it is not possible to uniquely estimate the full parameter vector $\widehat{\theta}$. At best, one can identify the corresponding level set
\[
\mathcal{M}_{\widehat{\vartheta}} 
= \bigl\{\theta \in B : \xi(\theta) = \xi(\widehat{\theta}) = \widehat{\vartheta}\bigr\}.
\]
In practice, estimating $\widehat{\vartheta}$ is itself a nontrivial inverse problem due to the non-injectivity of the algebraic mapping $\xi$. For the purpose of this study, we assume that an accurate estimate $\widehat{\vartheta}$ is available. The remaining task is then to recover biologically meaningful parameter values $\theta$ lying on the manifold $\mathcal{M}_{\widehat{\vartheta}}$.
Such a manifold may be disconnected. In particular, algebraic symmetries in the model, such as the exchangeability of $(\gamma_1, \gamma_2)$ and $(\beta_1, \beta_2)$, can induce multiple disjoint components corresponding to distinct but observationally equivalent parameter values.

We now demonstrate that Algorithm~\ref{alg:ptchmc} enables efficient sampling from disconnected implicit manifolds of the form $\mathcal{M}_{\widehat{\vartheta}}$. Without loss of generality, we choose the parameter domain $B$ as a product of intervals consistent with ranges reported in the epidemiological literature. We further assume a uniform prior $\mu_0$ over $B$, so that sampling from the posterior $\mu(\theta \mid y)$ restricted to $\mathcal{M}_{\widehat{\vartheta}}$ reduces to uniform sampling over the manifold.
We select the true parameter vector as
\[
\theta = \bigl(1.6\times 10^{-7},\, 1.3\times 10^{-7},\, 0.2,\, 0.1,\, 0.3\bigr),
\]
with total population size $4\times 10^6$ and initial conditions $I_1(0)=30$ and $I_2(0)=50$. The corresponding identifiable combination is given by
\[
\widehat{\vartheta} = \xi(\theta) = \bigl(0.3,\; 0.02,\; 9.67 \times 10^{-7},\; 1.0 \times 10^{6}\bigr).
\]
The inverse problem is thus formulated as characterizing the implicit algebraic manifold
\(
\mathcal{M}_{\widehat{\vartheta}} = \bigl\{ \theta \in B : \xi(\theta) = \widehat{\vartheta} \bigr\}.
\)
To mitigate numerical instability arising from singularities at \(\rho = 0\) and \(\beta_1 = \beta_2\), and to ensure that the constraint mapping \(\xi\) remains smooth over the parameter space \(B\), we apply a smooth regularization to the rational components. We introduce a small regularization parameter \(\delta = 1 \times 10^{-8}\) to form robust, differentiable approximations. The regularized constraint mapping is thus defined as
\[
\tilde{\xi}(\theta) = 
\left(
\gamma_1 + \gamma_2,\;
\gamma_1 \gamma_2,\;
\frac{(\beta_1 + \beta_2)\rho}{\rho^2 + \delta^2},\;
\frac{\rho(\gamma_1 - \gamma_2)(\beta_1 - \beta_2)}{(\beta_1 - \beta_2)^2 + \delta^2}
\right).
\]
By bounding the denominators strictly away from zero without introducing non-differentiable transitions, this modification prevents exploding gradients when \(\rho \rightarrow 0\) or \(\beta_1 \rightarrow \beta_2\). It ensures that \(\tilde{\xi}\) is a smooth embedding satisfying Assumption~\ref{ass:xi}. We then apply Algorithm~\ref{alg:ptchmc} to explore the manifold. The results are summarized in Figure~\ref{fig3}.
Figure~\ref{fig3}a shows a projected density of the implicit manifold on the $(\gamma_1,\rho)$ plane. The manifold $\mathcal{M}_{\widehat{\vartheta}} $ exhibits two disconnected components corresponding to $\gamma_1 = 0.1$ and $\gamma_1 = 0.2$. A projection of the relaxed distribution is also overlaid, illustrating that samples concentrate near the two components while maintaining sufficient mass between them to enable transitions across branches.
Figure~\ref{fig3}b presents the marginal distribution of $\gamma_1$, which clearly shows two dominant modes aligned with the two manifold components. Figure~\ref{fig3}c displays the projection of the manifold onto the $(\log_{10} \beta_1, \log_{10} \beta_2)$ plane, where two parallel disconnected structures are observed. In Figure~\ref{fig3}d, we report the proportion of samples visiting one connected component, denoted by $\widehat{p}(\mathcal{M}_A)$. The results indicate that the chain transitions between the two components with approximately equal probability, as expected due to symmetry.
These results demonstrate that the proposed algorithm successfully captures the global structure of the solution set, rather than being confined to a single local component. This highlights its effectiveness for sampling from implicitly defined, disconnected manifolds arising in non-identifiable inverse problems.
In settings where such analytical identifiable combinations are not available, this framework provides a promising direction for developing identifiability-aware sampling methods. In particular, one may combine standard inference on identifiable combinations with manifold-based exploration techniques to recover and characterize the corresponding solution sets in the original parameter space.

\begin{figure}[htbp]
\centering
\begin{subfigure}[b]{0.45\textwidth}
    \includegraphics[scale = .8]{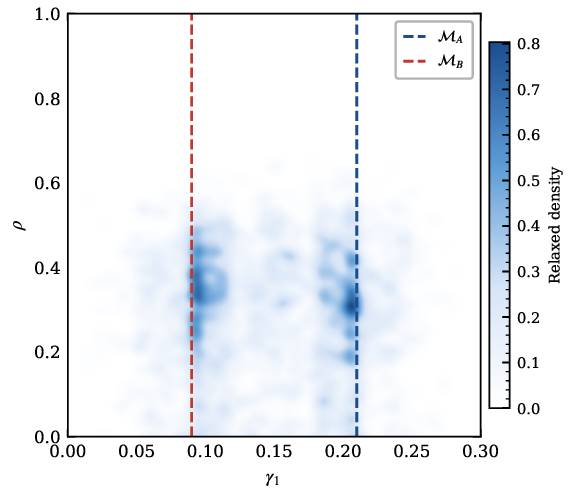}
    \caption{Projected density on the $(\gamma_1, \rho)$ plane}
    \label{fig3a}
\end{subfigure}%
\hfill
\begin{subfigure}[b]{0.05\textwidth}
\end{subfigure}%
\hfill
\begin{subfigure}[b]{0.45\textwidth}
    \includegraphics[scale = .8]{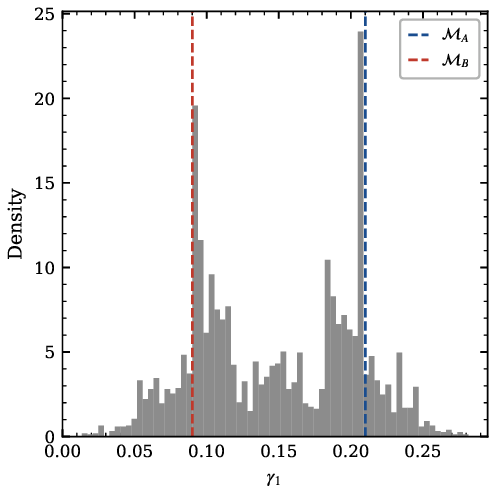}
    \caption{Marginal distribution of $\gamma_1$}
    \label{fig3b}
\end{subfigure}%

\begin{subfigure}[b]{0.45\textwidth}
    \includegraphics[scale = .8]{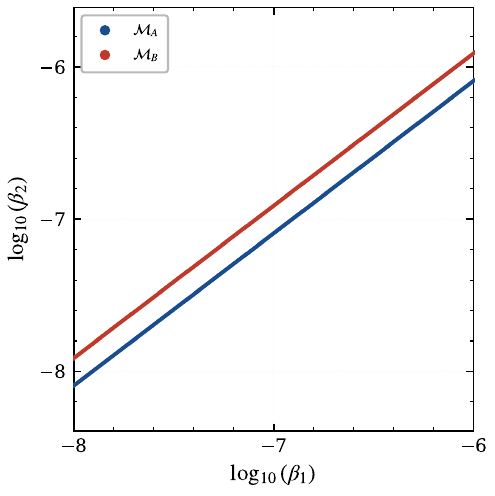}
    \caption{Projected $\mathcal{M}_{\widehat{\vartheta}} $ on the $(\log_{10} \beta_1, \log_{10} \beta_2)$ plane}
    \label{fig3c}
\end{subfigure}%
\hfill
\begin{subfigure}[b]{0.05\textwidth}
\end{subfigure}%
\hfill
\begin{subfigure}[b]{0.45\textwidth}
    \includegraphics[scale = .8]{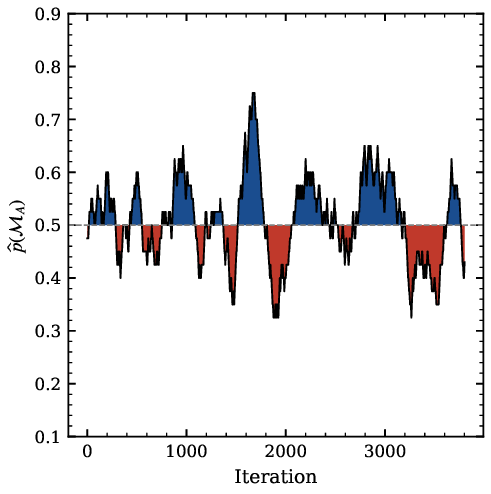}
    \caption{Proportion of samples visiting each branch}
    \label{fig3d}
\end{subfigure}%
\caption{Sampling results on the disconnected implicit manifold $\mathcal{M}_{\widehat{\vartheta}} $.}
\label{fig3}
\end{figure}

\subsection{Molecular Dynamics}
\label{C3}

The last numerical example shows the application of the proposed algorithm \ref{alg:ptchmc} to molecular dynamics (MD). MD simulations aim to sample molecular configurations from a Boltzmann distribution defined on a high-dimensional configuration space. When rigid geometric constraints, such as fixed bond lengths and bond angles, are imposed, the admissible configurations no longer occupy an open subset of $\mathbb{R}^d$, but instead lie on a lower-dimensional implicit manifold $\mathcal{M}$. Sampling from distributions supported on such constraint manifolds is a central challenge in computational chemistry and biophysics \citep{J28, J29}. In particular, algorithm ~\ref{alg:CHMC} provides an approach for sampling in this setting.
Beyond the difficulties introduced by constraints, an additional challenge arises when the constraint manifold is disconnected. This situation occurs naturally in stereochemistry, where molecular symmetry partitions the configuration space into topologically disjoint components corresponding to distinct enantiomeric (mirror-image) configurations. Since continuous rigid motions cannot connect these components without violating the constraints, standard CHMC samplers become confined to a single component. As a result, they fail to explore the full support of the target distribution, leading to biased estimates of thermodynamic quantities that depend on the global structure of the distribution.

The objective of this numerical experiment is to evaluate the performance of the proposed method on a disconnected constraint manifold arising from a standard MD stimulation setting. We demonstrate that the algorithm enables transitions between disconnected components and assess the accuracy of the estimated component probabilities against a known analytical reference.
To improve the efficiency of the exchange mechanism in a moderately high-dimensional configuration space, we employ multiple relaxed chains with distinct relaxation parameters $\epsilon$. Sampling efficiency and convergence are evaluated using the effective sample size (ESS) and the Gelman--Rubin diagnostic $\hat{R}$.
To this end, we construct a controlled benchmark inspired by the rigid geometry of a tetrahedral molecule. The resulting configuration space decomposes into exactly two disconnected components with a known probability ratio available in closed form.
Consider a tetrahedral configuration.
\[
  q = (q_1, q_2, q_3, q_4) \in \mathbb{R}^{12}, \qquad q_i \in \mathbb{R}^3.
\]
The configuration is invariant under global translation.  We eliminate this
redundancy by fixing the central atom at the origin, $q_1 = (0,0,0)$, so that
the remaining degrees of freedom are
\(
  q = (q_2, q_3, q_4) \in \mathbb{R}^9.
\)
Rigid tetrahedral geometry is enforced through six scalar constraints.  The three
bond-length constraints require each bond to have unit length
\[
\left\{
\begin{aligned}
  \xi_1(q) &= \|q_2\|^2 - 1 = 0, \\
  \xi_2(q) &= \|q_3\|^2 - 1 = 0, \\
  \xi_3(q) &= \|q_4\|^2 - 1 = 0.
\end{aligned}
\right.
\]
In a regular tetrahedron, the angle between any two bonds satisfies
$\theta = \arccos(-1/3) \approx 109.47^\circ$.
The three bond-angle constraints are therefore
\[
\left\{
\begin{aligned}
  \xi_4(q) &= q_2 \cdot q_3 + \tfrac{1}{3} = 0, \\
  \xi_5(q) &= q_2 \cdot q_4 + \tfrac{1}{3} = 0, \\
  \xi_6(q) &= q_3 \cdot q_4 + \tfrac{1}{3} = 0.
\end{aligned}
\right.
\]
Collecting these equations, we define the constraint map
$\xi(q) = (\xi_1(q), \ldots, \xi_6(q)) : \mathbb{R}^9 \to \mathbb{R}^6$,
and the implicit constraint manifold is then
\(
  \mathcal{M} = \{ q \in \mathbb{R}^9 : \xi(q) = 0 \}.
\)
Since six independent constraints are imposed on a nine-dimensional ambient
space, the manifold has dimension $\dim(\mathcal{M}) = 9 - 6 = 3$.
Geometrically, $\mathcal{M}$ parameterizes all rigid orientations of the
tetrahedral molecule with the central atom fixed at the origin, and it can be checked that such a manifold satisfies assumption \ref{ass:xi}.

The global structure of $\mathcal{M}$ is characterised by the signed volume of the tetrahedron formed by the three bond vectors,
\[
  V(q) = \det\!\begin{bmatrix} q_2 & q_3 & q_4 \end{bmatrix}.
\]
For any $q \in \mathcal{M}$, the absolute value of the volume is constant, $|V(q)| = V_0$, where $V_0 = \frac{\sqrt{2}}{12}$. The sign of $V(q)$ distinguishes the two enantiomers, and the manifold decomposes into two disjoint components,
\[
  \mathcal{M} = \mathcal{M}_+ \cup \mathcal{M}_-, \qquad 
  \mathcal{M}_\pm = \{ q \in \mathcal{M} : V(q) = \pm V_0 \}.
\]
Any continuous path connecting $\mathcal{M}_+$ to $\mathcal{M}_-$ must pass through configurations with $V(q)=0$, corresponding to coplanar bond vectors, which violate the tetrahedral constraints. Hence, $\mathcal{M}_+$ and $\mathcal{M}_-$ are topologically disconnected.
We consider the Boltzmann distribution on $\mathcal{M}$ as the target probability measure, defined as in equation \eqref{eq:target_measure},
\[
  \pi(q) \propto \exp(-\beta U(q))\, d\sigma(q),
\]
where $\beta=1$ and $d\sigma$ denotes the Riemannian volume measure on $\mathcal{M}$. To introduce asymmetry between the two components, we define a chiral potential
\(
  U(q) = \alpha V(q), 
\) for $\alpha > 0.$
Since $V(q)$ is constant on each component, the corresponding probabilities satisfy
\[
  P(\mathcal{M}_+) \propto e^{-\alpha V_0}, \qquad
  P(\mathcal{M}_-) \propto e^{\alpha V_0},
\]
and hence
\[
  \frac{P(\mathcal{M}_+)}{P(\mathcal{M}_-)} = \exp(-2\alpha V_0).
\]
Setting $\alpha = 6$ gives
\[
  \frac{P(\mathcal{M}_+)}{P(\mathcal{M}_-)} = \exp(-\sqrt{2}) \approx 0.243,
\]
so that $\mathcal{M}_-$ carries approximately $80.4\%$ of the total probability mass.
Standard CHMC evolves dynamics constrained to $\mathcal{M}$ and therefore preserves the sign of $V(q)$. As a result, trajectories cannot transition between $\mathcal{M}_+$ and $\mathcal{M}_-$, and the sampler remains confined to the initial component, leading to biased estimates.

To overcome topological barriers, we apply the Algorithm~\ref{alg:ptchmc}. The hard constraint $\xi(q)=0$ is relaxed using a parameter $\varepsilon > 0$, as defined in equation \eqref{eq:tubular_relaxation}. Larger values of $\varepsilon$ diffuse the distribution into the ambient space, smoothing the energy landscape and bridging the neighborhoods of $\mathcal{M}_+$ and $\mathcal{M}_-$. This relaxation enables the cold CHMC chain to traverse otherwise disconnected components.
To optimize exchange efficiency and facilitate global exploration, we employ a ladder of relaxation parameters, $\varepsilon \in \{0.05,\, 0.15,\, 0.30,\, 0.60\}$, with each level evolving independently via HMC adapted to the relaxed potential (Algorithm~\ref{alg:HMC}). Periodic replica exchanges between adjacent levels allow transitions discovered at higher relaxation scales to propagate back to the target level. This ensures the constrained chain effectively samples both $\mathcal{M}_+$ and $\mathcal{M}_-$.

The algorithm was executed for $10^6$ iterations. The cold chain, which targets the exact constrained distribution $\pi(q)$, serves as the primary sampler for posterior inference. Empirical results from the cold chain yield estimated component probabilities $\widehat{P}(\mathcal{M}_+) \approx 0.196$ and $\widehat{P}(\mathcal{M}_-) \approx 0.804$, with an estimated ratio $\widehat{P}(\mathcal{M}_+) / \widehat{P}(\mathcal{M}_-) = 0.244$ ($95\%$ CI $\pm 0.04$).
The estimated ratio closely agrees with the analytical value $P(\mathcal{M}_+)/P(\mathcal{M}_-) \approx 0.243$, confirming that the sampler accurately captures the global distribution without becoming trapped in isolated topological modes.
Sampling performance is evaluated comprehensively across all tempering levels. In particular, the ESS, the $\hat{R}$, and the exchange acceptance rates between adjacent levels ($\alpha$) are reported in Table~\ref{tab:sampling-efficiency}. These metrics demonstrate that the proposed algorithm achieves both high accuracy and computational efficiency when sampling from disconnected manifold components.

\begin{table}[h]
\centering
\caption{Sampling efficiency and exchange acceptance rates across tempering levels.}
\label{tab:sampling-efficiency}
\begin{tabular}{lcccc}
\toprule
\textbf{Level} & \textbf{ESS} & \textbf{ESS$/N$} & $\hat{R}$ & $\alpha$ \\
\midrule
CHMC & 4{,}735 & $4.7\times10^{-3}$ & 1.003 & 0.47 \\
HMC ($\epsilon = 0.05$) & 3{,}841 & $3.8\times10^{-3}$ & 1.003 & 0.33 \\
HMC ($\epsilon = 0.15$) & 3{,}107 & $3.1\times10^{-3}$ & 1.004 & 0.42 \\
HMC ($\epsilon = 0.30$) & 4{,}392 & $4.4\times10^{-3}$ & 1.002 & 0.36 \\
HMC ($\epsilon = 0.60$) & 6{,}815 & $6.8\times10^{-3}$ & 1.001 & --- \\
\bottomrule
\end{tabular}
\end{table}
The ESS results indicate that the proposed algorithm achieves satisfactory sampling efficiency, with consistently reasonable ESS and ESS$/N$ values across all tempering levels. 
Since the relaxed HMC dynamics can traverse near-coplanar configurations outside the constraint manifold $\mathcal{M}$, the cold CHMC chain, which targets the true distribution $\mu$, is able to explore disconnected components via replica exchange. This mechanism enables effective global exploration despite the underlying topological barriers.
exchange acceptance rates ranging from $0.36$ to $0.47$ indicate sufficient overlap between adjacent tempering levels, ensuring efficient information propagation across the hierarchy. In particular, the exchange acceptance probability between the CHMC chain and the coldest HMC chain ($\epsilon = 0.05$) is sufficiently high to support effective mixing. If needed, a smaller value of $\epsilon$ could be introduced to further increase this acceptance rate.
Moreover, $\hat{R} \leq 1.004$ for all chains indicates satisfactory convergence.

 This tetrahedral molecule example provides a controlled and analytically tractable test case for constrained sampling on disconnected manifolds. The known probability ratio $P(\mathcal{M}_+)/P(\mathcal{M}_-)$ enables rigorous quantitative validation of estimation accuracy, while the underlying geometric structure reflects configurations encountered in rigid molecular dynamics and stereochemistry. The results demonstrate that the proposed algorithm overcomes the topological trapping inherent in standard CHMC for molecular dynamics problems, delivering accurate and computationally efficient sampling of target distributions supported on disconnected constraint manifolds.

\section{Discussion}\label{sec6}
In this work, we proposed a replica exchange constrained Hamiltonian Monte Carlo framework for sampling probability measures supported on implicitly defined manifolds that may contain multiple disconnected components. The method couples an exact constrained chain evolving on the manifold with a relaxed auxiliary chain defined in a tubular neighborhood, and connects them through a geometrically informed exchange mechanism that preserves the correct target distribution.
The core theoretical contribution is a rigorous treatment of the exchange mechanism between a singular measure supported on a manifold and a relaxed measure supported on the ambient space. Unlike classical parallel tempering, which assumes all chains share a common reference measure, our setting requires a careful change-of-variables correction arising from the Jacobian of the exchange map. We derived a closed-form expression for this Jacobian in Theorem \ref{thm:jacobian_gram} and proposed a computationally tractable Gram determinant approximation, showing that the induced bias in total variation distance is of order $\sqrt{\varepsilon}$ (Theorem \ref{thm:gram_bias}). We further established detailed balance, irreducibility, and ergodicity of the full transition kernel in Theorem \ref{thm:pt_full}, providing a complete theoretical foundation for the algorithm. The three case studies demonstrate that the proposed algorithm enables transitions between disconnected components and recovers target distributions that standard CHMC fails to explore.

A fundamental trade-off throughout this method is the choice of the relaxation parameter $\varepsilon$. Small values of $\varepsilon$ concentrate the relaxed chain tightly around the manifold, reducing the approximation bias of the Gram correction and increasing the probability that proposed exchange states pass the reversibility check. However, it can also restricts the relaxed chain from traversing the regions of ambient space connecting distinct manifold components, thereby reducing the practical transition rate between disconnected branches. The multi-level tempering ladder provides a resolution, chains at larger relaxation scales explore the ambient space freely, while lower-temperature chains progressively refine this exploration toward the constrained target. Selecting an appropriate ladder of relaxation parameters remains a practically important calibration task, and adaptive schemes for tuning these parameters represent a natural direction for future work.

Several limitations of the current framework merit further discussion. First, the method depends on the existence and numerical tractability of the tubular neighborhood projection, which necessitates solving a nonlinear system at each proposed exchange step. Although Newton's method typically succeeds for smooth, well-conditioned constraints, the projection can become ill-conditioned near high-curvature regions of the manifold or when the constraint Jacobian approaches rank deficiency. Developing numerically robust projection schemes for these challenging regimes is an area for future research. Second, the current approach couples a single cold constrained chain with one or more relaxed chains. Extending this framework to accommodate multiple constrained chains, such as targeting distinct identifiable combinations in non-identifiable models, or adapting it to infinite-dimensional function spaces for PDE-constrained Bayesian inverse problems, would expand its practical applicability.

 Beyond the specific applications considered here, the proposed framework is broadly applicable to any setting in which holonomic constraints induce a disconnected configuration space. This includes problems in robotics and mechanism design, algebraic statistics, and the study of symmetry-induced degeneracies in physics models. The similarity between disconnected manifolds and multimodality suggests that insights from parallel tempering literature can be adapted to constrained manifold sampling, benefiting both research communities.

\bigskip
% \begin{center}
% {\large\bf SUPPLEMENTAL MATERIALS}
% \end{center}

% \begin{description}

% \item[Title:] Brief description. (file type)

% \item[R-package for  MYNEW routine:] R-package ÒMYNEWÓ containing code to perform the diagnostic methods described in the article. The package also contains all datasets used as examples in the article. (GNU zipped tar file)

% \item[HIV data set:] Data set used in the illustration of MYNEW method in Section~ 3.2. (.txt file)

% \end{description}

\bibliographystyle{plainnat} 

% This points to your cas-refs.bib file
\bibliography{cas-refs}

\end{document}